\documentclass[pdflatex,sn-mathphys-num]{sn-jnl}
\usepackage{graphicx}%
\usepackage{subcaption}
\usepackage{multirow}%
\usepackage{amsmath,amssymb,amsfonts}%
\usepackage{amsthm}%
\usepackage{mathrsfs}%
\usepackage[title]{appendix}%
\usepackage{textcomp}%
\usepackage{manyfoot}%
\usepackage{booktabs}%
\usepackage{listings}%
\usepackage{stmaryrd}
\usepackage[ruled,vlined]{algorithm2e}
\usepackage{float}

\DeclareMathOperator*{\argmax}{arg\,max}
\DeclareMathOperator*{\argmin}{arg\,min}

\theoremstyle{thmstyleone}%
\newtheorem{theorem}{Theorem}[section]
\newtheorem{lemma}[theorem]{Lemma}

\newtheorem{corollary}[theorem]{Corollary}

\newtheorem{remark}[theorem]{Remark}

\newcommand{\D}{D_{\mathrm{KL}}^{\pi}}

\begin{document}

\title[Information-theoretic minimax and submodular optimization algorithms for multivariate Markov chains]{Information-theoretic minimax and submodular optimization algorithms for multivariate Markov chains}

\author[1]{\fnm{Zheyuan} \sur{Lai}}\email{\href{mailto:zheyuan_lai@u.nus.edu}{\texttt{zheyuan\_lai@u.nus.edu}}}

\author*[1]{\fnm{Michael C.H.} \sur{Choi}}\email{\href{mailto:mchchoi@nus.edu.sg}{\texttt{mchchoi@nus.edu.sg}}}

\affil[1]{\orgdiv{Department of Statistics and Data Science}, \orgname{National University of Singapore}, \orgaddress{\street{Level 7, 6 Science Drive 2}, \postcode{117546}, \country{Singapore}}}


\abstract{We study an information-theoretic minimax problem for finite multivariate Markov chains on $d$-dimensional product state spaces. Given a family $\mathcal B=\{P_1,\ldots,P_n\}$ of $\pi$-stationary transition matrices and a class $\mathcal F = \mathcal{F}(\mathbf{S})$ of factorizable models induced by a partition $\mathbf S$ of the coordinate set $\llbracket d \rrbracket$, we seek to minimize the worst-case information loss by analyzing
\begin{align}\label{eq:abs}
    \min_{Q\in\mathcal F}\max_{P\in\mathcal B} \D(P\|Q),
\end{align}
where $\D(P\|Q)$ is the $\pi$-weighted KL divergence from $Q$ to $P$. We recast \eqref{eq:abs} into a concave maximization problem over the $n$-probability-simplex via strong duality and Pythagorean identities that we derive. This leads us to formulate \eqref{eq:abs} into an information-theoretic game and show that a mixed strategy Nash equilibrium always exists; and propose a projected subgradient algorithm to approximately solve \eqref{eq:abs} with provable guarantee. By transforming \eqref{eq:abs} into an orthant submodular function in $\mathbf{S}$, this motivates us to consider a max-min-max submodular optimization problem and investigate a two-layer subgradient–greedy procedure to approximately solve this generalization. Numerical experiments for Markov chains on the Curie–Weiss and Bernoulli–Laplace models illustrate the practicality of these proposed algorithms and reveals sparse optimal structures in these examples.}


\maketitle

\section{Introduction}
Multivariate Markov chains on product spaces $\mathcal X=\mathcal X^{(1)}\times\ldots\times\mathcal X^{(d)}$ with $d \in \mathbb{N}$ arise naturally throughout stochastic modeling, Markov chain Monte Carlo (MCMC), and interacting particle systems. In high dimensions when $d$ is large, it is natural, both for analysis and for algorithm design, to approximate a complex transition matrix $P$ by a simpler model that \emph{factorizes} across groups of coordinates. This paper develops an information-theoretic framework, associated structure theorems, and algorithms for selecting such factorizations and for aggregating multiple candidate Markov chains in a robust or minimax sense.

\paragraph{Related works.} This manuscript centers on the following three main threads: information projections of multivariate Markov chains, minimax information aggregation, and submodular optimization over partition. In the literature, \cite{choi2024geometry} views factorization as minimizing the KL divergence between the original chain and the set of factorizable Markov chains; \cite{lacker2025independent} introduces the independent projection of diffusion processes through the lens of relative entropy minimization in the space of product measures. On the topic of minimax information aggregation, \cite{Haussler97,Gushchin06} study minimax optimization under KL divergence and $f$-divergences of probability measures, while \cite{HafezKolahi2022Information} analyzes minimax excess risk as a zero-sum game between a learner and Nature. As for (robust) submodular optimization over partition, \cite{nemhauser1978analysis} and \cite{ward2016maximizing} propose greedy-based algorithms when the partition set function is submodular or $k$-submodular; \cite{orlin2018robust} handles robust submodular optimization with bi-level optimization; \cite{bogunovic2017robust} proposes novel algorithm with non-uniform partitions; \cite{staib2019robust} applies continuous submodular functions to address the robust budget allocation problem.

We proceed to describe the contributions and the organizations of the paper in the rest of this Section.

\paragraph{Problem setup.}
We first fix notations and quickly recall several established results in submodularity and information projections of Markov chains in Section \ref{sec:prelim}, followed by introducing the information-theoretic minimax problem in Section \ref{sec:minimax}.

Precisely, we denote $\mathcal{L}(\mathcal{X})$ to be the set of transition matrices on $\mathcal{X}$. Let $\mathcal B=\{P_1,\ldots,P_n\}\subset\mathcal L(\mathcal X)$ be a family of $\pi$-stationary transition matrices on $\mathcal{X}$ and let $\mathbf S=(S_1,\ldots,S_m)$ be a partition of $\llbracket d\rrbracket$, where we write $\llbracket d \rrbracket := \{1, 2, \ldots, d\}$. We consider the class of factorizable transition matrices with respect to the partition $\mathbf{S}$
\[
\mathcal{F} = \mathcal F(\mathbf S) :=\{Q \in \mathcal{L}(\mathcal{X});~Q=Q^{(S_1)}\otimes\cdots\otimes Q^{(S_m)}\},
\]
and the associated minimax approximation problem
\begin{equation}\label{eq:intro-minimax}
\min_{Q\in\mathcal F}\max_{P\in\mathcal B} \D(P\|Q).
\end{equation}
Here, we denote $P^{(S_j)}$ to be the projection of $P$ onto the coordinate set $S_j$, which we call the keep-$S_j$-in transition matrix, while $\otimes_{j=1}^m$ is the $m$-fold tensor product. Problem \eqref{eq:intro-minimax} considers minimizing the worst-case information loss when replacing any $P\in\mathcal B$ by a factorizable proxy $Q$ with respect to $\mathbf{S}$.

\paragraph{Averaging, information projection and a two-person game.}
In Section \ref{sec:minimax}, through strong duality and Pythagorean identities, we establish that
\begin{equation}\label{eq:intro-dual}
    \min_{Q\in\mathcal F(\mathbf S)}\max_{P\in\mathcal B} \D (P\|Q)
    = \max_{\mathbf w\in\mathcal S_n}\sum_{i=1}^n w_i\, \D \left(P_i \| \otimes_{j=1}^m \overline P(\mathbf w)^{(S_j)}\right),
\end{equation}
which transforms \eqref{eq:intro-minimax} into a concave maximization problem over the $n$-probability-simplex $\mathcal{S}_n$, where $\overline{P}(\mathbf{w}) := \sum_{i=1}^n w_i P_i$ is the $\mathbf{w}$-weighted average of the matrices in $\mathcal{B}$. 

We interpret the minimax problem \eqref{eq:intro-minimax} in a two-person zero-sum game in Section \ref{sec:game}, and prove that a mixed strategy Nash equilibrium always exists. This generalizes the reversiblization entropy games in \cite{choi2023markov} to the context of factorizations of multivariate Markov chains as in this paper.

\paragraph{Orthant submodularity and optimal partition.}
In Section \ref{sec:maxminmax}, for fixed $\mathbf{w} \in \mathcal{S}_n$, we prove that the map
\[
m^{\llbracket d \rrbracket} \ni \mathbf{S} \mapsto \sum_{i=1}^n w_i \D (P_i \| (\otimes_{j=1}^{m - 1} \overline{P}(\mathbf{w})^{(S_j)}) \otimes \overline{P}(\mathbf{w})^{(-\mathrm{supp}(\mathbf{S}))})
\]
is orthant submodular, where we define $\mathrm{supp}(\mathbf{S}) := \cup_{i=1}^{m-1} S_i.$ This result enables greedy-style algorithms with provable guarantees when designing partitions \cite{lai2025information}.

\paragraph{Algorithms.}
(i) \emph{A projected subgradient algorithm.} In Section \ref{sec:subgrad}, We derive explicit supergradients of the concave dual \eqref{eq:intro-dual} in $\mathbf w$. We propose and analyze a subgradient algorithm and prove $\mathcal O(t^{-1/2})$ convergence in objective value, where $t$ is the number of iterations of the algorithm.

(ii) \emph{A two-layer subgradient-greedy algorithm.} In Section \ref{sec:maxminmax}, we consider the problem of jointly optimizing over both $\mathbf S$ and $\mathbf w$. Specifically, we cast a max–min–max problem whose inner value admits \eqref{eq:intro-dual}. For fixed $\mathbf S$, we iterate $\mathbf w$ by projected subgradients; holding $\mathbf w$ fixed, we exploit orthant submodularity to perform a generalized distorted greedy update on $\mathbf S$, yielding a practical alternating procedure with a provable lower bound.

\paragraph{Experiments.}
In Section \ref{sec:num}, we give numerical experiments on the Curie–Weiss and Bernoulli–Laplace models. We investigate multivariate Markov chains in these models and observe (a) \emph{sparse} optimal mixtures that put mass on a few extrema (e.g., base $P$ and an accelerated or lazy variant) and (b) interpretable partitions that capture dominant dependence while controlling the worst-case KL loss. These case studies corroborate the theory and highlight the practicality of the proposed algorithms.

\section{Preliminaries}\label{sec:prelim}

\subsection{Projection and averaging of multivariate Markov chains}
We consider a finite $d$-dimensional state space described by $\mathcal{X}= \mathcal{X}^{(1)} \times \ldots \times \mathcal{X}^{(d)}$. For $S\subseteq \llbracket d \rrbracket$, we write $\mathcal{X}^{(S)} = \times_{i \in S} \mathcal{X}^{(i)}$ and $\mathcal{X}^{(-S)} = \times_{i \notin S} \mathcal{X}^{(i)}$, which are subsets of $\mathcal{X}$. We denote $\mathcal{L}(\mathcal{X})$ to be the set of transition matrices on $\mathcal{X}$, and $\mathcal{P}(\mathcal{X}) = \{\pi;~\min_{x \in \mathcal{X}} \pi(x) > 0, \sum_{x} \pi(x) = 1\}$ to be the set of probability masses with full support on $\mathcal{X}$. We say that $P \in \mathcal{L}(\mathcal{X})$ is $\pi$-\textbf{stationary} with $\pi \in \mathcal{P}(\mathcal{X})$ if it satisfies $\pi = \pi P$.

We then recall the definition of the \textbf{tensor product} of transition matrices and probability masses, see e.g. Exercise 12.6 of \cite{levin2017markov}. Define, for $M_l \in \mathcal{L}(\mathcal{X}^{(l)})$, $\pi_l \in \mathcal{P}(\mathcal{X}^{(l)})$, $x^l,y^l \in \mathcal{X}^{(l)}$ for $l \in \{i,j\}, i \neq j \in \llbracket d \rrbracket,$
\begin{align*}
    (M_i \otimes M_j)((x^i, x^j), (y^i, y^j)) &:= M_i(x^i, y^i) M_j(x^j, y^j), \\
    (\pi_i \otimes \pi_j)(x^i, x^j) &:= \pi_i(x^i) \pi_j(x^j).
\end{align*}

To define the projection operations, we recall the definition of keep-$S$-in and leave-$S$-out matrices of a given transition probability matrix $P$, see Section 2.2 of \cite{choi2024geometry}. For \(\pi \in \mathcal{P}(\mathcal{X})\), \(P \in \mathcal{L} (\mathcal{X})\), \(S \subseteq \llbracket d \rrbracket\), and any \((x^{(-S)}, y^{(-S)}) \in \mathcal{X}^{(-S)} \times \mathcal{X}^{(-S)}\), we define the $\textbf{leave}$-$S$-\textbf{out} transition matrix with respect to $\pi$ to be $P_\pi^{(-S)}$ with entries given by
    \[
    P_\pi^{(-S)}(x^{(-S)}, y^{(-S)}) :=
    \frac{\sum_{(x^{(S)}, y^{(S)}) \in \mathcal{X}^{(S)} \times \mathcal{X}^{(S)}} \pi(x^1, \dots, x^d) P((x^1, \dots, x^d), (y^1, \dots, y^d))}
    {\sum_{x^{(S)} \in \mathcal{X}^{(S)}} \pi(x^1, \dots, x^d)}.
    \]
The \textbf{keep}-\(S\)-\textbf{in} transition matrix of \(P\) with respect to \(\pi\) is 
    \[
    P_\pi^{(S)} := P_\pi^{(-\llbracket d \rrbracket \setminus S)} \in \mathcal{L}(\mathcal{X}^{(S)}).
    \]
When \(P\) is \(\pi\)-stationary, we omit the subscript $\pi$ and write directly \(P^{(-S)}, P^{(S)}\). We also apply the convention of $P^{(\emptyset)} = P^{(-\llbracket d \rrbracket)} = 1$.

We then define the \textbf{averaging operation} $\overline{P}(\mathbf{w})$ of a transition probability matrix $P$. We define $\mathcal{S}_n$ as the $n$-probability-simplex such that\begin{align*}
    \mathcal{S}_n = \left\{\mathbf{w} = (w_1, \ldots, w_n) \in \mathbb{R}_+^n;~ \sum_{i=1}^n w_i = 1 \right\}.
\end{align*}
Given a set of $\pi$-stationary transition probability matrices $\mathcal{B} = \{P_1, \ldots, P_n\}$, we define the transition probability matrix weighted by $\mathbf{w} = (w_1, \ldots, w_n) \in \mathcal{S}_n$ as $\overline{P} (\mathbf{w})$ by  \begin{align*}
    \overline{P} = \overline{P}(\mathbf{w}) := \sum_{i=1}^n w_i P_i.
\end{align*}
We see that $\overline{P}$ is also $\pi$-stationary because \begin{align*}
    \pi \overline{P} = \pi \left(\sum_{i=1}^n w_i P_i\right) = \sum_{i=1}^n w_i (\pi P_i) = \sum_{i=1}^n w_i \pi = \pi.
\end{align*}

We project each $P_i$ onto $S\in 2^{\llbracket d \rrbracket}$ and denote the weighted projection as 
\begin{align*}
    \overline{P}(S, \mathbf{w}) := \sum_{i=1}^n w_i P_i^{(S)}.
\end{align*}
As a result, we have
\begin{align*}
    \overline{P}^{(S)} = \left(\sum_{i=1}^n w_i P_i\right)^{(S)} = \sum_{i=1}^n w_i P_i^{(S)} = \overline{P}(S, \mathbf{w}),
\end{align*}
which means that the averaging operation commutes with the projection operation.

\subsection{Some information-theoretic results in Markov chain theory}
We first recall the Shannon entropy of a probability distribution and the entropy rate of a transition probability matrix, see Section 1 of \cite{polyanskiy2025information}. For probability distribution $\pi$ on $\mathcal{X}$, its \textbf{Shannon entropy} is defined as \begin{align*}
    H(\pi) := - \sum_{x\in\mathcal{X}} \pi(x) \ln{\pi(x)},
\end{align*}
while for $\pi$-stationary $P \in \mathcal{L}(\mathcal{X})$, the \textbf{entropy rate} of $P$ is defined as \begin{align*}
    H(P) := -\sum_{x \in \mathcal{X}} \sum_{y \in \mathcal{X}} \pi(x) P(x, y) \ln{P(x, y)},
\end{align*}
where the standard convention of $0\ln{0} := 0$ applies.

We then recall the KL divergence between Markov chains (see Definition 2.1 of \cite{choi2024geometry}). For given $\pi \in \mathcal{P}(\mathcal{X})$ and transition matrices $M, L \in \mathcal{L}(\mathcal{X})$, we define the \textbf{KL divergence} from $L$ to $M$ with respect to $\pi$ as \begin{align*}
    \D (M \| L) := \sum_{x \in \mathcal{X}} \pi (x) \sum_{y\in \mathcal{X}} M(x, y) \ln{\frac{M(x, y)}{L(x, y)}}
\end{align*}
where the convention of $0 \ln \frac{0}{a} := 0$ applies for $a \in [0,1]$.

We then prove a Pythagorean identity related to the averaging operation and the KL divergence of transition matrices. 

\begin{lemma}\label{lem:pyth_KL_m}
    For given $\mathbf{w} \in \mathcal{S}_n$, $\pi \in \mathcal{P}(\mathcal{X})$, $P_i, Q \in \mathcal{L}(\mathcal{X})$ for $i \in \llbracket n \rrbracket$ where $P_i$ are all $\pi$-stationary, we choose mutually disjoint sets $S_1, \ldots, S_m$ with $\sqcup_{i=1}^m S_i = \llbracket d \rrbracket$, and the following identity holds:
    \begin{align}\label{eq:pyth_KL_m}
        \sum_{i=1}^n w_i \D (P_i \| \otimes_{j=1}^m Q^{(S_j)}) = \sum_{i=1}^n w_i \D (P_i \| \otimes_{j=1}^m \overline{P}^{(S_j)}) + \sum_{j=1}^m D_\mathrm{KL}^{\pi^{(S_j)}} (\overline{P}^{(S_j)} \| Q^{(S_j)}).
    \end{align}
    In particular, we have the following minimization result:
    \begin{align*}
        \min_{Q;~ Q = \otimes_{j=1}^m Q^{(S_j)}} \sum_{i=1}^n w_i \D(P_i \| Q) = \sum_{i=1}^n w_i \D (P_i \| \otimes_{j=1}^m \overline{P}^{(S_j)}).
    \end{align*}
\end{lemma}

\begin{proof}
    Inspired by Theorem 2.22 of \cite{choi2024geometry}, we note that
    \begin{align*}
        &\quad \sum_{i=1}^n w_i \D (P_i \| \otimes_{j=1}^m Q^{(S_j)}) \\
        &= \sum_{i=1}^n w_i \D (P_i \| \otimes_{j=1}^m \overline{P}^{(S_j)}) + \sum_{i=1}^n w_i \sum_{x, y} \pi(x) P_i(x, y) \ln{\frac{\otimes_{j=1}^m \overline{P}^{(S_j)}(x, y)}{\otimes_{j=1}^m Q^{(S_j)} (x, y)}} \\
        &= \sum_{i=1}^n w_i \D (P_i \| \otimes_{j=1}^m \overline{P}^{(S_j)}) + \sum_{j=1}^m \sum_{i=1}^n w_i \sum_{x^{(S_j)}, y^{(S_j)}} \pi^{(S_j)} (x^{(S_j)})P_i^{(S_j)} (x^{(S_j)}, y^{(S_j)}) \ln{\frac{\overline{P}^{(S_j)}(x^{(S_j)}, y^{(S_j)})}{Q^{(S_j)}(x^{(S_j)}, y^{(S_j)})}} \\
        &= \sum_{i=1}^n w_i \D (P_i \| \otimes_{j=1}^m \overline{P}^{(S_j)}) + \sum_{j=1}^m D_\mathrm{KL}^{\pi^{(S_j)}} (\overline{P}^{(S_j)} \| Q^{(S_j)}),
    \end{align*}
    where the last equality comes from the fact that the averaging and projection operation commutes. As a consequence, since \begin{align*}
        D_\mathrm{KL}^{\pi^{(S_j)}} (\overline{P}^{(S_j)} \| Q^{(S_j)}) \geq 0,
    \end{align*}
    we therefore see that
    \begin{align*}
        \min_{Q;~ Q = \otimes_{j=1}^m Q^{(S_j)}} \sum_{i=1}^n w_i \D(P_i \| Q) = \sum_{i=1}^n w_i \D (P_i \| \otimes_{j=1}^m \overline{P}^{(S_j)}).
    \end{align*}
\end{proof}

As a corollary, in the special case of $m = 2$ with $S_1 = S$, $S_2 = \llbracket d \rrbracket \backslash S$, we see that
\begin{corollary}\label{lem:pyth_KL}
    For given $\mathbf{w} \in \mathcal{S}_n$, $\pi \in \mathcal{P}(\mathcal{X})$, $P_i, Q \in \mathcal{L}(\mathcal{X})$ for $i \in \llbracket n \rrbracket$ where $P_i$ are all $\pi$-stationary, $S \in 2^{\llbracket d \rrbracket}$, the following identity holds:
    \begin{align}\label{eq:pyth_KL}
        \sum_{i=1}^n w_i \D(P_i \| Q^{(S)} \otimes Q^{(-S)}) = \sum_{i=1}^n w_i \D (P_i \| \overline{P}^{(S)} \otimes \overline{P}^{(-S)}) + D_\mathrm{KL}^{\pi^{(S)}} (\overline{P}^{(S)} \| Q^{(S)}) + D_\mathrm{KL}^{\pi^{(-S)}} (\overline{P}^{(-S)} \| Q^{(-S)}).
    \end{align}
    In particular, we have the following minimization result: \begin{align*}
        \min_{Q;~ Q = Q^{(S)} \otimes Q^{(-S)}} \sum_{i=1}^n w_i \D (P_i \| Q) = \sum_{i=1}^n w_i \D (P_i \| \overline{P}^{(S)} \otimes \overline{P}^{(-S)}).
    \end{align*}
\end{corollary}

\subsection{Definition and examples of submodularity}
We first recall the definition of a submodular function (Section 14 of \cite{korte2008combinatorial}) and its generalization to $k$-submodularity. Given a finite nonempty ground set \( U \) as the fixed base domain, i.e. $U = \llbracket d \rrbracket$ in the context of this paper, a set function \( f : 2^U \to \mathbb{R} \)  defined on subsets of \( U \) is called \textbf{submodular} if for all \( S, T \subseteq U \),\begin{align*}
    f(S) + f(T) \geq f(S \cap T) + f(S \cup T).
\end{align*}

A multivariate generalization of submodularity is known as $k$-submodularity \cite{ene2022streaming} where $k \in \mathbb{N}$. Let \( f : (k + 1)^U \to \mathbb{R} \) be a set function. The function \( f \) is said to be \( k \)-\textbf{submodular} if  
\[
f(\mathbf{S}) + f(\mathbf{T}) \geq f(\mathbf{S} \sqcap \mathbf{T}) + f(\mathbf{S} \sqcup \mathbf{T}) \quad \forall \, \mathbf{S}, \mathbf{T} \in (k + 1)^U,
\]
where \(\mathbf{S} \sqcap \mathbf{T}\) is the \( k \)-tuple whose \( i \)-th set is \( S_i \cap T_i \) and \(\mathbf{S} \sqcup \mathbf{T}\) is the \( k \)-tuple whose \( i \)-th set is \((S_i \cup T_i) \setminus \left(\bigcup_{j \neq i} (S_j \cup T_j)\right)\). In particular, when $k=1$, an $1$-submodular function is equivalent to a submodular function.

We proceed to recall the definition of orthant submodularity \cite{ene2022streaming}.
For $\mathbf{S} = (S_1, \ldots, S_k),\mathbf T = (T_1, \ldots, T_k) \in (k+1)^U$, let \(\Delta_{e, i} f(\mathbf{S})\) be the marginal gain of adding \(e\) to the \(i\)-th set of \(\mathbf{S}\):
\begin{align*}
    \Delta_{e, i} f(\mathbf{S}) := f(S_1, \ldots, S_i \cup \{e\}, \ldots, S_k) - f(S_1, \ldots, S_i, \ldots, S_k).
\end{align*}
A function \( f \) is said to be \textbf{orthant submodular} if
\begin{align*}
    \Delta_{e,i}f(\mathbf S) \geq \Delta_{e,i}f(\mathbf T)
\end{align*}
for all \( i \in \llbracket k\rrbracket \) and \( \mathbf S, \mathbf T \in (k+1)^U \) such that \( \mathbf S \preceq \mathbf T \), \( e \notin \mathrm{supp}(\mathbf T) \), where $\mathbf{S} \preceq \mathbf{T}$ if and only if $S_i \subseteq T_i$, $\forall i \in \llbracket k \rrbracket$.

We then show some examples of submodular structures that arise in the information theory of Markov chains.

\begin{theorem}[Submodularity of some information-theoretic functions in Markov chain theory]\label{thm:submod_mc}
Let $\mathbf{w} \in \mathcal{S}_n$, \( S \subseteq \llbracket d\rrbracket \), \( P, P_i \in \mathcal{L}(\mathcal{X}) \) be  $\pi$-stationary transition matrices for $i \in \llbracket n \rrbracket$. We have
\begin{enumerate}
    \item\label{it:submod_ent} (Submodularity of the entropy rate of \( P \)) The mapping \( S \mapsto H(P^{(S)}) \) is submodular.
    \item\label{it:submod_KL} (Submodularity of the distance to \( (S, \llbracket d \rrbracket \backslash S) \)-factorizability of \( P \)) The mapping \( S \mapsto \D(P\|P^{(S)} \otimes P^{(-S)}) \) is submodular.
    \item\label{it:submod_ent_w} (Submodularity of the entropy rate of $\overline{P}$) The mapping $S \mapsto H(\overline{P}^{(S)})$ is submodular.
    \item\label{it:submod_KL_w} (Submodularity of the weighted distance to \( (S, \llbracket d \rrbracket \backslash S) \)-factorizability of $\mathcal{B}$) The mapping $S \mapsto \sum_{i=1}^n w_i \D (P_i \| \overline{P}^{(S)} \otimes \overline{P}^{(-S)})$ is submodular.
\end{enumerate}
\end{theorem}

\begin{proof}
    From Proposition 2.33 of \cite{choi2024geometry}, item \eqref{it:submod_ent} and item \eqref{it:submod_KL} hold. Since the map $S \mapsto H(P^{(S)})$ is submodular, the map $S \mapsto H(\overline{P}^{(S)})$ is submodular since $\overline{P}^{(S)}$ is the projection of $\overline{P}$ onto subset $S$, which proves item~\eqref{it:submod_ent_w}. Since \begin{align*}
        \sum_{i=1}^n w_i \D (P_i \| \overline{P}^{(S)} \otimes \overline{P}^{(-S)}) = H(\overline{P}^{(S)}) + H(\overline{P}^{(-S)}) - \sum_{i=1}^n w_i H(P_i),
    \end{align*}
    we can conclude that $S \mapsto \sum_{i=1}^n w_i \D (P_i \| \overline{P}^{(S)} \otimes \overline{P}^{(-S)})$ is submodular because both the map $S \mapsto H(\overline{P}^{(S)})$ and the map $S \mapsto H(\overline{P}^{(-S)})$ are submodular (by Lemma 2.1 of \cite{lai2025information}).
\end{proof}

\section{The minimax optimization problem}\label{sec:minimax}
We denote a feasible set $\mathcal{F}$, the set of factorizable transition matrices with respect to a partition $\mathbf{S}$: 
\begin{align*}
    \mathcal{F} = \mathcal{F}(\mathbf{S}) := \{Q \in \mathcal{L}(\mathcal{X});~ \mathbf{S} = (S_1, \ldots, S_m) \in (m+1)^{\llbracket d \rrbracket},~ Q = Q^{(S_1)}\otimes \ldots \otimes Q^{(S_m)}\}.
\end{align*}

We are interested in the following minimax optimization problem 
\begin{align}\label{eq:minmax_KL}
    \min_{Q \in \mathcal{F}} \max_{P\in \mathcal{B}} \D(P\|Q).
\end{align}
In words, we seek to find an optimal factorizable $Q \in \mathcal{F}$ that minimize the worst-case information loss in approximating members of $\mathcal{B}$.

Since $\mathcal{F}$ is not a convex set, we denote $$\mathcal{M} := \{M \in \mathbb{R}^{|\mathcal{X}| \times |\mathcal{X}|}\}$$ as the set of matrices on the state space $\mathcal{X}$ and study the weighted geometric mean and the following set: 
\begin{align*}
    \mathcal{A} := \left\{A \in \mathcal{M};~ \exists \, l \in \mathbb{N}, \mathbf{c} \in \mathcal{S}_l \text{ s.t. } A(x, y) = \sum_{i=1}^l c_i \log Q_i (x, y),~ \forall x, y;~ Q_i \in \mathcal{F},~ \forall i \in \llbracket l \rrbracket \right\}.
\end{align*}

\begin{lemma}
    The set $\mathcal{A}$ is convex.
\end{lemma}

\begin{proof}
    We choose $A, B \in \mathcal{A}$ such that there exists $\mathbf{c} \in \mathcal{S}_l$, $\mathbf{d} \in \mathcal{S}_k$, $Q_i,R_j \in \mathcal{F}$ for $i \in \llbracket l \rrbracket, j \in \llbracket k \rrbracket$ and for all $x,y$ we have
    \begin{align*}
        A(x, y) = \sum_{i=1}^l c_i \log Q_i (x, y),~ B(x, y) = \sum_{i=1}^k d_i \log R_i (x, y).
    \end{align*}
    We choose $\alpha \in [0, 1]$ and calculate that
    \begin{align*}
        \quad \alpha A(x, y) + (1 - \alpha) B(x, y) = \sum_{i=1}^l \alpha c_i \log Q_i (x, y) + \sum_{i=1}^k (1 - \alpha) d_i \log R_i (x, y).
    \end{align*}
    We thus conclude that $\alpha A + (1 - \alpha) B \in \mathcal{A}$, and hence $\mathcal{A}$ is convex.
\end{proof}
We define the \textbf{elementwise exponential} of a matrix $M \in \mathcal{M}$ to be $\exp{M}$, that is, for all $x,y \in \mathcal{X}$,
\begin{align*}
    \exp{M} (x,y) := e^{M(x,y)}.
\end{align*}
For given $P \in \mathcal{L}(\mathcal{X})$, we define the \textbf{generalized KL divergence} from the non-negative and not necessarily stochastic matrix $\exp{A}$ to $P$ to be
\begin{align*}
    \widetilde{D}_\mathrm{KL}^\pi (P \| A) &:= \sum_{x, y} \pi(x) P(x, y) \log\frac{P(x, y)}{\exp{A}(x, y)}\\
    &= \sum_{x, y} \pi(x) P(x, y) \log P(x, y) - \sum_{x, y} \pi(x) P(x, y) A(x, y),
\end{align*}
which is linear in $A$, hence the map $\mathcal{A} \ni A \mapsto \widetilde{D}_\mathrm{KL}^\pi (P \| A)$ is convex.

We study the following minimax optimization problem \begin{align}\label{eq:minimax_KL_gen}
    \min_{A \in \mathcal{A}} \max_{P \in \mathcal{B}} \widetilde{D}_\mathrm{KL}^\pi (P \| A).
\end{align}
By \cite[equation (3)]{candan2020chebyshev}, we can reformulate problem \eqref{eq:minimax_KL_gen} as \begin{align}\label{eq:minimax_KL_gen_rev}
    & \min_{A \in \mathcal{A},~ r} r\\
    & \text{s.t. } \quad \widetilde{D}_\mathrm{KL}^\pi (P_i \| A) \leq r,~ \forall i\in \llbracket n \rrbracket, \nonumber
\end{align}
which is a constrained convex minimization problem.

Comparing problem \eqref{eq:minmax_KL} with problem \eqref{eq:minimax_KL_gen}, we note that for every $Q \in \mathcal{F}$, we can define an associated $A \in \mathcal{A}$ such that $A(x,y) = \log Q(x,y)$, and hence we have the following inequality: 
\begin{align}\label{eq:KL_gen_ineq}
    \min_{Q \in \mathcal{F}} \max_{P \in \mathcal{B}} \D (P \| Q) \geq \min_{A \in \mathcal{A}} \max_{P \in \mathcal{B}} \widetilde{D}_\mathrm{KL}^\pi (P \| A).
\end{align}

Suppose $A\in \mathcal{A}$ such that $\exp A(x, y) = \prod_{i=1}^l Q_i(x, y)^{c_i}$ for any $x, y$, we then show a Pythagorean identity based on the proof of Theorem 2.22 of \cite{choi2024geometry}:
\begin{align}\label{eq:pyth_gen_KL}
    \widetilde{D}_\mathrm{KL}^\pi (P \| A) &= \sum_{x, y} \pi(x) P(x, y) \log{\frac{P(x, y)}{\prod_{i=1}^l Q_i(x, y)^{c_i}}}\nonumber \\
    &= \D (P \| \otimes_{i=1}^m P^{(S_i)}) + \sum_{x, y} \pi(x) P(x, y) \log{\frac{\otimes_{i=1}^m P^{(S_i)}(x, y)}{\prod_{j=1}^l Q_j(x, y)^{c_j}}} \nonumber \\
    &= \D (P \| \otimes_{i=1}^m P^{(S_i)}) + \sum_{i=1}^m \sum_{j=1}^l c_j \D (P^{(S_i)} \| Q_j^{(S_i)}) \geq \widetilde{D}_\mathrm{KL}^\pi (P \| A^*),
\end{align}
where $A^* = A^*(S_1,\ldots,S_m,P) \in \mathcal{A}$ is defined to be 
\begin{align*}
    A^* (x, y) := \log(\otimes_{i=1}^m P^{(S_i)} (x, y)).
\end{align*}

Inspired by \eqref{eq:pyth_gen_KL} and Lemma~\ref{lem:pyth_KL_m}, for given $\mathbf{w} \in \mathcal{S}_n$, we show a weighted version of Pythagorean identity for generalized KL divergence: 
\begin{align}\label{eq:pyth_gen_KL_m}
    \sum_{i=1}^n w_i \widetilde{D}_\mathrm{KL}^\pi (P_i \| A) &= \sum_{i=1}^n w_i \sum_{x, y} \pi(x) P_i(x, y) \log{\frac{P_i(x, y)}{\prod_{k=1}^l Q_k(x, y)^{c_k}}} \nonumber \\
    &= \sum_{i=1}^n w_i \D (P_i \| \otimes_{j=1}^m \overline{P}^{(S_j)}) + \sum_{i=1}^n w_i \sum_{x, y} \pi(x) P_i(x, y) \log{\frac{\otimes_{j=1}^m \overline{P}^{(S_j)} (x, y)}{\prod_{k=1}^l Q_k (x, y)^{c_k}}} \nonumber \\
    &= \sum_{i=1}^n w_i \D (P_i \| \otimes_{j=1}^m \overline{P}^{(S_j)}) + \sum_{j=1}^m \sum_{k=1}^l c_k D_\mathrm{KL}^{\pi^{(S_j)}} (\overline{P}^{(S_j)} \| Q_k^{(S_j)}) \\
    &\geq \sum_{i=1}^n w_i \widetilde{D}_\mathrm{KL}^\pi (P_i \| A^*_n (\mathbf{w})), \nonumber
\end{align}
where $A^*_n (\mathbf{w}) = A^*_n(\mathbf{w},S_1,\ldots,S_m,\mathcal{B}) \in \mathcal{A}$ is defined to be, for all $x,y \in \mathcal{X}$,
\begin{align*}
    A^*_n (x, y) := \log(\otimes_{j=1}^m \overline{P}^{(S_j)}) (x, y).
\end{align*}
In the special case that $n = 1$, we recover that $A^*_1 = A^*$.

For the problem \eqref{eq:minimax_KL_gen_rev}, we denote the Lagrangian $L: \mathbb{R}_+ \times \mathcal{A} \times \mathbb{R}_+^n$ to be 
\begin{align}\label{eq:LagrangeL}
    L(r, A, \mathbf{w}) := r + \sum_{i=1}^n w_i (\widetilde{D}_\mathrm{KL}^\pi (P_i \| A) - r),
\end{align}
where $\mathbf{w}$ is the associated Lagrangian multiplier.

From the Pythagorean identity \eqref{eq:pyth_gen_KL_m}, the dual problem of \eqref{eq:minimax_KL_gen_rev} can be written as \begin{align}\label{eq:minimax_KL_gen_dual}
    \max_{\mathbf{w} \in \mathbb{R}_+^n} \min_{r \geq 0,~ A \in \mathcal{A}} L(r, A, \mathbf{w}) &= \max_{\mathbf{w} \in \mathcal{S}_n} \min_{A \in \mathcal{A}} \sum_{i=1}^n w_i \widetilde{D}_\mathrm{KL}^\pi (P_i \| A) = \max_{\mathbf{w} \in \mathcal{S}_n} \sum_{i=1}^n w_i \widetilde{D}_\mathrm{KL}^\pi (P_i \| A^*_n (\mathbf{w})).
\end{align}
The main results in this section are that strong duality holds for problem \eqref{eq:minimax_KL_gen_rev}, and problem \eqref{eq:minmax_KL} and \eqref{eq:minimax_KL_gen} are equivalent. We write the results in the following theorem.

\begin{theorem}\label{thm:equiv}
    \begin{enumerate}
        \item\label{it:sd} The strong duality holds for problem \eqref{eq:minimax_KL_gen_rev} and there exists $\mathbf{w}^* \in \mathcal{S}_n$ such that \begin{align*}
            \min_{A \in \mathcal{A}} \max_{P \in \mathcal{B}} \widetilde{D}_\mathrm{KL}^\pi (P \| A) = \max_{\mathbf{w} \in \mathcal{S}_n} \sum_{i=1}^n w_i \widetilde{D}_\mathrm{KL}^\pi (P_i \| A^*_n(\mathbf{w})) = \sum_{i=1}^n w_i^* \widetilde{D}_\mathrm{KL}^\pi (P_i \| A^*_n (\mathbf{w}^*)).
        \end{align*}
        \item\label{it:cs} Suppose the pair $(A, r) \in \mathcal{A} \times \mathbb{R}_+$ minimizes the primal problem \eqref{eq:minimax_KL_gen_rev} and $\mathbf{w}^* \in \mathcal{S}_n$ maximizes the dual problem \eqref{eq:minimax_KL_gen_dual}, then the following complementary slackness results hold: for $i \in \llbracket n \rrbracket$, we have \begin{align*}
            \widetilde{D}_\mathrm{KL}^\pi (P_i \| A) \begin{cases}
                = r, \quad \text{ if } w_i^* > 0; \\
                \leq r, \quad \text{ if } w_i^* = 0.
            \end{cases}
        \end{align*}
        \item\label{it:equiv} Problems \eqref{eq:minmax_KL} and \eqref{eq:minimax_KL_gen} are equivalent, i.e. 
        \begin{align*}
            \min_{Q \in \mathcal{F}} \max_{P \in \mathcal{B}} \D (P \| Q) = \min_{A \in \mathcal{A}} \max_{P \in \mathcal{B}} \widetilde{D}_\mathrm{KL}^\pi (P \| A).
        \end{align*}
        \item\label{it:opt} The same $\mathbf{w}^* \in \mathcal{S}_n$ from item \eqref{it:sd} satisfies
        \begin{align*}
            \min_{Q \in \mathcal{F}} \max_{P \in \mathcal{B}} \D (P \| Q) = \max_{\mathbf{w} \in \mathcal{S}_n} \sum_{i=1}^n w_i \D (P_i \| \otimes_{k=1}^m \overline{P}(\mathbf{w})^{(S_k)}) = \sum_{i=1}^n w_i^* \D (P_i \| \otimes_{k=1}^m \overline{P}(\mathbf{w}^*)^{(S_k)}).
        \end{align*}

        \item\label{it:L_conc} The map 
        \begin{align*}
            \mathcal{S}_n \ni \mathbf{w} \mapsto \sum_{i=1}^n w_i \D (P_i \| \otimes_{k=1}^m \overline{P}(\mathbf{w})^{(S_k)})
        \end{align*}
        is concave in $\mathbf{w}$.
    \end{enumerate}
\end{theorem}

\begin{proof}
We first show item \eqref{it:sd}, i.e., strong duality holds for problem \eqref{eq:minimax_KL_gen_rev}. We shall show that the Slater's qualification is verified (see Section 5.2.3 of \cite{boyd2004convex} and Appendix A of \cite{beck2017first}), which requires that the constraints in \eqref{eq:minimax_KL_gen_rev} are strictly feasible. We take any $A$ and \begin{align*}
    r = \max_{i \in \llbracket n \rrbracket} \widetilde{D}_\mathrm{KL}^\pi (P_i \| A) + 1 > \widetilde{D}_\mathrm{KL}^\pi (P_l \| A),~ \forall l\in \llbracket n \rrbracket,
\end{align*}
hence the strong duality holds. Therefore we have \begin{align*}
    \min_{A \in \mathcal{A}} \max_{P \in \mathcal{B}} \widetilde{D}_\mathrm{KL}^\pi (P \| A) = \max_{\mathbf{w} \in \mathcal{S}_n} \sum_{i=1}^n w_i \widetilde{D}_\mathrm{KL}^\pi (P_i \| A^*_n(\mathbf{w})) = \sum_{i=1}^n w_i^* \widetilde{D}_\mathrm{KL}^\pi (P_i \| A_n^* (\mathbf{w}^*)).
\end{align*}
As the strong duality in item \eqref{it:sd} holds, by Section 5.5.2 of \cite{boyd2004convex}, the \emph{complementary slackness} condition holds, i.e. \begin{align*}
    w_i^* (\widetilde{D}_\mathrm{KL}^\pi (P_i \| A) - r) = 0,
\end{align*}
which is equivalent to 
    \begin{align*}
        \widetilde{D}_\mathrm{KL}^\pi (P_i \| A) \begin{cases}
            = r, \quad \text{ if } w_i^* > 0; \\
            \leq r, \quad \text{ if } w_i^* = 0,
        \end{cases}
    \end{align*}
for all $i \in \llbracket n \rrbracket$, hence it proves item \eqref{it:cs}.

We proceed to prove item \eqref{it:equiv}. Let $j \in \llbracket n \rrbracket$ be an index where $w^*_j > 0$, we want to show
\begin{align*}
    \widetilde{D}_\mathrm{KL}^\pi (P_j \| A^*_n (\mathbf{w}^*)) = \max_{l\in \llbracket n \rrbracket} \widetilde{D}_\mathrm{KL}^\pi (P_l \| A^*_n(\mathbf{w}^*)).
\end{align*}
As it is clear to see that $\widetilde{D}_\mathrm{KL}^\pi (P_j \| A^*_n (\mathbf{w}^*)) \leq \max_{l\in \llbracket n \rrbracket} \widetilde{D}_\mathrm{KL}^\pi (P_l \| A^*_n(\mathbf{w}^*))$, we then assume that
\begin{align*}
    \widetilde{D}_\mathrm{KL}^\pi (P_j \| A^*_n (\mathbf{w}^*)) < \max_{l\in \llbracket n \rrbracket} \widetilde{D}_\mathrm{KL}^\pi (P_l \| A^*_n(\mathbf{w}^*)).
\end{align*}
That is, there exists an index $l^*$ such that \begin{align*}
    \widetilde{D}^\pi_\mathrm{KL}(P_j \| A^*_n(\mathbf{w}^*)) < \widetilde{D}^\pi_\mathrm{KL}(P_{l^*} \| A^*_n(\mathbf{w}^*)).
\end{align*}
By strong duality, we have $w^*_{l^*} = 0$, then by complementary slackness in item \eqref{it:cs}, we have \begin{align*}
    \widetilde{D}_\mathrm{KL}^\pi (P_{l^*} \| A^*_n (\mathbf{w}^*)) \leq r = \widetilde{D}_\mathrm{KL}^\pi (P_j \| A^*_n (\mathbf{w}^*)) < \widetilde{D}_\mathrm{KL}^\pi (P_{l^*} \| A^*_n (\mathbf{w}^*)),
\end{align*}
which leads to a contradiction. It therefore yields \begin{align*}
    \widetilde{D}_\mathrm{KL}^\pi (P_j \| A^*_n (\mathbf{w}^*)) = \max_{l\in \llbracket n \rrbracket} \widetilde{D}_\mathrm{KL}^\pi (P_l \| A^*_n(\mathbf{w}^*)).
\end{align*}
By recalling the definition of generalized KL divergence and \eqref{eq:KL_gen_ineq}, we have
\begin{align*}
    \min_{Q \in \mathcal{F}} \max_{P \in \mathcal{B}} \D (P \| Q) &\geq \min_{A \in \mathcal{A}} \max_{P \in \mathcal{B}} \widetilde{D}_\mathrm{KL}^\pi (P \| A) = \max_{\mathbf{w} \in \mathcal{S}_n} \sum_{i=1}^n w_i \widetilde{D}_\mathrm{KL}^\pi (P_i \| A^*_n (\mathbf{w})) \\
    &= \max_{l\in \llbracket n \rrbracket} \widetilde{D}_\mathrm{KL}^\pi (P_l \| A^*_n(\mathbf{w}^*)) = \widetilde{D}_\mathrm{KL}^\pi (P_j \| A^*_n(\mathbf{w}^*)) \\
    &= \max_{P \in \mathcal{B}} \D (P \| \otimes_{k=1}^m \overline{P}(\mathbf{w}^*)^{(S_k)}) \geq \min_{Q \in \mathcal{F}} \max_{P \in \mathcal{B}} \D (P \| Q),
\end{align*}
therefore we obtain \begin{align*}
    \min_{Q \in \mathcal{F}} \max_{P \in \mathcal{B}} \D (P \| Q) = \min_{A \in \mathcal{A}} \max_{P \in \mathcal{B}} \widetilde{D}_\mathrm{KL}^\pi (P \| A),
\end{align*}
hence problem \eqref{eq:minmax_KL} and problem \eqref{eq:minimax_KL_gen} are equivalent. Therefore, for the $\mathbf{w}^* \in \mathcal{S}_n$ in item \eqref{it:sd}, we have \begin{align*}
    \min_{Q \in \mathcal{F}} \max_{P \in \mathcal{B}} \D(P \| Q) = \min_{A \in \mathcal{A}} \max_{P \in \mathcal{B}} \widetilde{D}_\mathrm{KL}^\pi (P \| A) &= \sum_{i=1}^n w_i^* \widetilde{D}_\mathrm{KL}^\pi (P_i \| A^*_n (\mathbf{w}^*))\\ &= \sum_{i=1}^n w_i^* \D (P_i \| \otimes_{k=1}^m \overline{P}(\mathbf{w}^*)^{(S_k)}),
\end{align*}
which proves item \eqref{it:opt}.

We then show item \eqref{it:L_conc}. From \eqref{eq:minimax_KL_gen_dual}, we have \begin{align*}
    \sum_{i=1}^n w_i \D (P_i \| \otimes_{k=1}^m \overline{P}(\mathbf{w})^{(S_k)}) = \sum_{i=1}^n w_i \widetilde{D}_\mathrm{KL}^\pi (P_i \| A_n^*) = \min_{r \geq 0,~ A\in \mathcal{A}} L(r, A, \mathbf{w}),
\end{align*}
hence the map \begin{align*}
    \mathcal{S}_n \ni \mathbf{w} \mapsto \sum_{i=1}^n w_i \D (P_i \| \otimes_{k=1}^m \overline{P}(\mathbf{w})^{(S_k)})
\end{align*}
is concave since it is the Lagrangian dual function of problem \eqref{eq:minimax_KL_gen_rev} (see Section 5.1.2 of \cite{boyd2004convex}).
\end{proof}

\section{An information-theoretic game}\label{sec:game}

Inspired by the reversiblization entropy games in \cite{choi2023markov}, we cast the minimax problem as a two–player zero–sum game between Nature and a probabilist. Nature chooses a transition probability matrix $P \in \mathcal{B}$, while the probabilist chooses an approximating factorizable transition matrix $Q \in \mathcal{F} = \mathcal{F}(\mathbf{S})$. The payoff from the probabilist to Nature is the KL divergence $\D(P\|Q)$, which Nature aims to \emph{maximize} while the probabilist aims to \emph{minimize}. 

In the \emph{pure strategy game}, Nature selects a single $P\in\mathcal{B}$ and the probabilist selects a single $Q\in\mathcal{F}$. In the \emph{mixed strategy game}, Nature is permitted to randomize over $\mathcal{B}$ according to a probability distribution $\mu\in\mathcal{P}(\mathcal{B})$ (which corresponds to a weight vector $\mathbf{w} \in \mathcal{S}_n$), while the probabilist still chooses a single $Q\in\mathcal{F}$.

We adapt the following notations for some related minimax and maximin values: \begin{align*}
    \overline{V} = \overline{V}(\mathbf{S}, \mathcal{B}) &:= \min_{Q \in \mathcal{F}} \max_{\mu \in \mathcal{P}(\mathcal{B})} \int_\mathcal{B} \D (P \| Q) \mu(\mathrm{d}P), \\
    \underline{V} = \underline{V}(\mathbf{S}, \mathcal{B}) &:= \max_{\mu \in \mathcal{P}(\mathcal{B})} \min_{Q \in \mathcal{F}} \int_\mathcal{B} \D (P \| Q) \mu(\mathrm{d}P), \\
    \overline{v} = \overline{v}(\mathbf{S}, \mathcal{B}) &:= \min_{Q \in \mathcal{F}} \max_{P \in \mathcal{B}} \D (P \| Q), \\
    \underline{v} = \underline{v}(\mathbf{S}, \mathcal{B}) &:= \max_{P \in \mathcal{B}} \min_{Q \in \mathcal{F}} \D (P \| Q).
\end{align*}

From item~\eqref{it:opt} of Theorem~\ref{thm:equiv}, the pure-strategy minimax value $\overline{v}$ is equivalent to the dual problem: \begin{align}\label{eq:over_v}
    \overline{v} = \min_{Q \in \mathcal{F}} \max_{P \in \mathcal{B}} \D (P \| Q) = \max_{\mathbf{w} \in \mathcal{S}_n} \sum_{i=1}^n w_i \D (P_i \| \otimes_{j=1}^m \overline{P}(\mathbf{w})^{(S_j)}).
\end{align}

The following theorem establishes the existence of a mixed-strategy Nash equilibrium (see Section 3 of \cite{osborne1994course}), which is a foundational result in game theory.

\begin{theorem}[Existence of mixed strategy Nash equilibrium]\label{thm:nash}
Consider the two-person mixed strategy game with respect to parameters $(\mathbf{S}, \mathcal{B})$,
\begin{enumerate}
    \item\label{it:nash_exist} The mixed strategy Nash equilibrium always exists. That is, the value of the game is well-defined and given by
    \begin{align*}
        \overline{V}(\mathbf{S}, \mathcal{B}) = \underline{V}(\mathbf{S}, \mathcal{B}) = \max_{\mathbf{w} \in \mathcal{S}_n} \sum_{i=1}^n w_i \D (P_i \| \otimes_{j=1}^m \overline{P}(\mathbf{w})^{(S_j)}).
    \end{align*}
    \item\label{it:nash_Qmu} The mixed strategy Nash equilibrium is attained at $(Q^*, \mu^*)$, where $\mu^*$ is represented by the optimal weight vector $\mathbf{w}^* \in \mathcal{S}_n$ and $Q^*$ is the information projection of the corresponding weighted average $\overline{P}(\mathbf{w}^*)$ onto $\mathcal{F}$, i.e. \begin{align*}
        Q^* = \otimes_{j=1}^m \overline{P}(\mathbf{w}^*)^{(S_j)}.
    \end{align*}
\end{enumerate}
\end{theorem}

\begin{proof}
We first show existence in item~\eqref{it:nash_exist}. By Proposition 3.10 of \cite{choi2023markov}, we have the standard minimax inequalities $\overline{v}(\mathbf{S}, \mathcal{B}) \geq \overline{V}(\mathbf{S}, \mathcal{B}) \geq \underline{V}(\mathbf{S}, \mathcal{B})$. We can also establish a lower bound for $\underline{V}$ by restricting Nature's strategy space from all probability measures $\mathcal{P}(\mathcal{B})$ to the simplex of finite measures $\mathcal{S}_n$:
\begin{align*}
    \underline{V} = \underline{V}(\mathbf{S}, \mathcal{B}) &= \max_{\mu \in \mathcal{P}(\mathcal{B})} \min_{Q \in \mathcal{F}} \int_\mathcal{B} \D (P \| Q) \mu(\mathrm{d}P) \\
    &\geq \max_{\mathbf{w} \in \mathcal{S}_n} \min_{Q \in \mathcal{F}} \sum_{i=1}^n w_i \D (P_i \| Q) \\
    &= \max_{\mathbf{w} \in \mathcal{S}_n} \sum_{i=1}^n w_i \D (P_i \| \otimes_{j=1}^m \overline{P}(\mathbf{w})^{(S_j)}) = \overline{v},
\end{align*}
where the second last equality comes from Lemma~\ref{lem:pyth_KL_m} and the final equality comes from \eqref{eq:over_v}. We have thus shown the chain of inequalities $\overline{v} \geq \overline{V} \geq \underline{V} \geq \overline{v}$, which enforces equality throughout. This implies $\overline{V} = \underline{V}$, confirming that the mixed-strategy Nash equilibrium exists.

Item~\eqref{it:nash_Qmu} follows from item~\eqref{it:nash_exist}. At the mixed-strategy Nash equilibrium, the pair of optimal strategies $(Q^*, \mu^*)$ is composed of Nature's optimal strategy $\mu^*$, which is represented by the optimal weight vector $\mathbf{w}^* \in \mathcal{S}_n$, and the probabilist's optimal pure strategy $Q^* \in \mathcal{F}$. Nature's strategy $\mathbf{w}^*$ is the solution to the dual maximization problem as in item~\eqref{it:opt} of Theorem~\ref{thm:equiv}, identifying the ``worst-case'' mixture in $\mathcal{B}$. In response to this specific mixture, the probabilist's unique best response $Q^*$ is the information projection of the corresponding weighted average model $\overline{P}(\mathbf{w}^*)$ onto the set of factorizable $\mathcal{F}$, which is explicitly given by $Q^* = \otimes_{j=1}^m \overline{P}(\mathbf{w}^*)^{(S_j)}$.
\end{proof}

\section{A projected subgradient algorithm}\label{sec:subgrad}
From Theorem \ref{thm:equiv}, since problems \eqref{eq:minmax_KL} and \eqref{eq:minimax_KL_gen} are equivalent (item \eqref{it:equiv}), hence by item \eqref{it:opt}, it suffices to solving the following convex minimization problem: 
\begin{align}\label{eq:subg}
    \min_{\mathbf{w} \in \mathcal{S}_n} \quad h(\mathbf{w}), 
\end{align}
where $h(\mathbf{w}) = - \sum_{i=1}^n w_i \D(P_i \| \otimes_{k=1}^m \overline{P}(\mathbf{w})^{(S_k)})$ is convex from item \eqref{it:L_conc}. We now compute a subgradient of $h$, through which we aim to propose a projected subgradient algorithm with theoretical guarantee.

\begin{theorem}[Subgradient of $h$ and an upper bound of its $l^2$-norm]\label{thm:subg}
    A subgradient of $h$ at $\mathbf{v} \in \mathcal{S}_n$ is given by $\mathbf{g} = \mathbf{g}(\mathbf{v}) = (g_1, \ldots, g_n) \in \mathbb{R}^n$, where for all $i \in \llbracket n \rrbracket$, we have \begin{align*}
        g_i = g_i(\mathbf{v}) = \D(P_n \| \otimes_{k=1}^m \overline{P}(\mathbf{v})^{(S_k)}) - \D (P_i \| \otimes_{k=1}^m \overline{P}(\mathbf{v})^{(S_k)}).
    \end{align*}
    The subgradient $\mathbf{g}$ satisfies that, for all $\mathbf{w}, \mathbf{v} \in \mathcal{S}_n$, \begin{align*}
        h(\mathbf{w}) \geq h(\mathbf{v}) + \sum_{i=1}^n g_i \cdot (w_i - v_i).
    \end{align*}
    Moreover, the $l^2$-norm of $\mathbf{g}(\mathbf{v})$ is bounded above by \begin{align*}
        \|\mathbf{g}\|_2^2 = \sum_{i=1}^n g_i^2 \leq n \left( |\mathcal{X}| \sup_{\mathbf{v} \in \mathcal{S}_n;~ i \in \llbracket n \rrbracket;~ P_i (x, y) > 0 } P_i(x, y) \ln{\frac{P_i(x, y)}{\otimes_{k=1}^m \overline{P}(\mathbf{v})^{(S_k)} (x, y)}}\right)^2 := B.
    \end{align*}
\end{theorem}

\begin{proof}
    By the Pythagorean identity (Lemma~\ref{lem:pyth_KL_m}), we have \begin{align*}
        \sum_{i=1}^n w_i \D (P_i \| \otimes_{k=1}^m \overline{P}(\mathbf{w})^{(S_k)}) \leq \sum_{i=1}^n w_i \D (P_i \| \otimes_{k=1}^m \overline{P}(\mathbf{v})^{(S_k)})
    \end{align*}
    for any $\mathbf{w}, \mathbf{v} \in \mathcal{S}_n$. Hence, \begin{align*}
        h(\mathbf{w}) - h(\mathbf{v}) &= -\sum_{i=1}^n w_i \D (P_i \| \otimes_{k=1}^m \overline{P}(\mathbf{w})^{(S_k)}) + \sum_{i=1}^n v_i \D (P_i \| \otimes_{k=1}^m \overline{P}(\mathbf{v})^{(S_k)}) \\
        &\geq -\sum_{i=1}^n (w_i - v_i) \D (P_i \| \otimes_{k=1}^m \overline{P}(\mathbf{v})^{(S_k)}) \\
        &= -\sum_{i=1}^n (w_i - v_i) \D (P_i \| \otimes_{k=1}^m \overline{P}(\mathbf{v})^{(S_k)}) + \sum_{i=1}^n (w_i - v_i) \D (P_n \| \otimes_{k=1}^m \overline{P}(\mathbf{v})^{(S_k)}) \\
        &= \sum_{i=1}^n (w_i - v_i) g_i,
    \end{align*}
    where the second last equation holds because $\mathbf{w}, \mathbf{v} \in \mathcal{S}_n$, and hence $\sum_{i=1}^n (w_i - v_i) = 0$.
    
    We proceed to prove the upper bound on the $l^2$-norm. We first show the upper bound of the KL divergence term: \begin{align*}
        \D (P_i \| \otimes_{k=1}^m \overline{P}(\mathbf{v})^{(S_k)}) &= \sum_{x \in \mathcal{X}} \pi(x) \sum_{y \in \mathcal{X}} P_i (x, y) \ln{\frac{P_i (x, y)}{\otimes_{k=1}^m \overline{P}(\mathbf{v})^{(S_k)} (x, y)}} \\
        &\leq |\mathcal{X}| \sup_{\mathbf{v} \in \mathcal{S}_n;~ i \in \llbracket n \rrbracket;~ P_i(x, y)>0} P_i (x, y) \ln{\frac{P_i (x, y)}{\otimes_{k=1}^m \overline{P}(\mathbf{v})^{(S_k)} (x, y)}} = \sqrt{\frac{B}{n}},
    \end{align*}
    where $|\mathcal{X}|$ denotes the cardinality of the state space $\mathcal{X}$. Then we have \begin{align*}
        \|\mathbf{g}\|_2^2 = \sum_{i=1}^n g_i^2 &\leq \sum_{i=1}^n \max \left\{\D (P_n \| \otimes_{k=1}^m \overline{P}(\mathbf{v})^{(S_k)})^2, \D (P_i \| \otimes_{k=1}^m \overline{P}(\mathbf{v})^{(S_k)})^2 \right\} \\
        &\leq n \max_{l \in \llbracket n \rrbracket} \D (P_l \| \otimes_{k=1}^m \overline{P}(\mathbf{v})^{(S_k)})^2 \leq n \cdot \sqrt{\frac{B}{n}}^2 = B.
    \end{align*}
\end{proof}

Inspired by Algorithm 1 of \cite{choi2023markov}, we propose a projected subgradient algorithm to solve problem \eqref{eq:subg}. In Algorithm~\ref{alg:subg}, we conduct the projected subgradient algorithm for $t$ iterations. At each iteration, we first update the weight parameters via subgradient,
\begin{align*}
    \mathbf{v}^{(i)} = \mathbf{w}^{(i-1)} - \eta \cdot \mathbf{g}(\mathbf{w}^{(i-1)}),
\end{align*}
where $\eta > 0$ is the stepsize of the algorithm while we take $\mathbf{g}$ as in Theorem~\ref{thm:subg}, the subgradient of $h$. In the second step, the updated weight $\mathbf{v}^{(i)}$ is to be projected onto the $n$-probability-simplex $\mathcal{S}_n$, i.e. \begin{align*}
    \mathbf{w}^{(i)} = \argmin_{\mathbf{w} \in \mathcal{S}_n} \|\mathbf{w} - \mathbf{v}^{(i)}\|_2^2,
\end{align*} 
which can be accomplished by existing projection algorithms onto a simplex (see e.g. \cite{condat2016fast}).
Note that the subgradient algorithm is not a descent algorithm, hence the monotonicity of $h(\mathbf{w})$ among different iterations is not guaranteed, see Section \ref{subsec:num_subg} for examples.

\SetKwInput{Input}{Input}
\SetKwInput{Output}{Output}

\begin{algorithm}[H]
\caption{A projected subgradient algorithm to solve problem \eqref{eq:subg}}
\label{alg:subg}
\Input{Initial weight value $\mathbf{w}^{(0)} \in \mathcal{S}_n$, set $\{P_i\}_{i=1}^n$, target distribution $\pi$, stepsize $\eta > 0$, and number of iterations $t$.}
\Output{The sequence $\bigl(\mathbf{w}^{(i)}\bigr)_{i=1}^t$.}
\For{$i = 1,2,\ldots,t$}{
  $\mathbf{v}^{(i)} = \mathbf{w}^{(i-1)} - \eta \cdot \mathbf{g}(\mathbf{w}^{(i-1)})$ \texttt{ // Update via subgradient descent}
  
  $\mathbf{w}^{(i)} = \argmin_{\mathbf{w} \in \mathcal{S}_n} \|\mathbf{w} - \mathbf{v}^{(i)}\|_2^2$ \texttt{ // Project onto $\mathcal{S}_n$}
}
\end{algorithm}

The rest of the section is devoted to providing a theoretical guarantee for Algorithm~\ref{alg:subg}. We first prove an upper bound of Algorithm \ref{alg:subg}.

\begin{theorem}[Upper bound of Algorithm~\ref{alg:subg}]\label{thm:subg_ub}
    Consider Algorithm~\ref{alg:subg} with its outputs $(\mathbf{w}^{(i)})_{i=1}^t$, we have 
    \begin{align*}
        h(\overline{\mathbf{w}}^t) - h(\mathbf{w}^*) \leq \frac{n}{2 \eta t} + \frac{\eta B}{2},
    \end{align*}
    where $\overline{\mathbf{w}}^t = \frac{1}{t} \sum_{i=1}^t \mathbf{w}^{(i)}$ and $\mathbf{w}^*$ is the optimal solution to problem \eqref{eq:subg}. Furthermore, if we choose constant stepsize $\eta = \sqrt{\frac{n}{Bt}}$, we have \begin{align*}
        h(\overline{\mathbf{w}}^t) - h(\mathbf{w}^*) \leq \sqrt{\frac{nB}{t}}.
    \end{align*}
    In addition, given any $\epsilon > 0$, if we further choose \begin{align*}
        t = \left\lceil \frac{nB}{\epsilon^2} \right\rceil,
    \end{align*}
    then we can reach an $\epsilon$-close value to $h(\mathbf{w}^*)$ such that \begin{align*}
        h(\overline{\mathbf{w}}^t) - h(\mathbf{w}^*) \leq \epsilon.
    \end{align*}
\end{theorem}

\begin{proof}
    For all $i \in \llbracket t \rrbracket$, due to projection, we have 
    \begin{align*}
        \|\mathbf{w}^{(i+1)} - \mathbf{w}^*\|_2^2 &\leq \|\mathbf{v}^{(i+1)} - \mathbf{w}^*\|_2^2 = \| \mathbf{w}^{(i)} - \eta \cdot \mathbf{g}(\mathbf{w}^{(i)}) - \mathbf{w}^*\|_2^2 \\
        &= \| \mathbf{w}^{(i)} - \mathbf{w}^* \|_2^2 + \eta^2 \|\mathbf{g}(\mathbf{w}^{(i)})\|^2 - 2\eta \mathbf{g}(\mathbf{w}^{(i)})(\mathbf{w}^{(i)} - \mathbf{w}^*)\\
        &\leq \|\mathbf{w}^{(i)} - \mathbf{w}^*\|_2^2 + \eta^2 B - 2\eta \mathbf{g}(\mathbf{w}^{(i)})(\mathbf{w}^{(i)} - \mathbf{w}^*),
    \end{align*}
    where the last inequality come from the upper bound in Theorem~\ref{thm:subg}. We then apply the definition of subgradient $\mathbf{g}$ in Theorem~\ref{thm:subg}, and it leads to \begin{align*}
        h(\mathbf{w}^{(i)}) - h(\mathbf{w}^*) &\leq \mathbf{g}(\mathbf{w}^{(i)}) \cdot (\mathbf{w}^{(i)} - \mathbf{w}^*) \\
        &\leq \frac{1}{2\eta} \left(\|\mathbf{w}^{(i)} - \mathbf{w}^*\|_2^2 - \|\mathbf{w}^{(i+1)} - \mathbf{w}^*\|_2^2\right) + \frac{\eta B}{2}.
    \end{align*}
    We then take summation on $i$ from $1$ to $t$ and obtain
    \begin{align*}
        \sum_{i=1}^t (h(\mathbf{w}^{(i)}) - h(\mathbf{w}^*)) &\leq \frac{1}{2\eta} \left(\|\mathbf{w}^{(1)} - \mathbf{w}^*\|_2^2 - \|\mathbf{w}^{(t+1)} - \mathbf{w}^*\|_2^2\right) + \frac{\eta Bt}{2} \\
        &\leq \frac{1}{2\eta} \|\mathbf{w}^{(i)} - \mathbf{w}^*\|_2^2 + \frac{\eta Bt}{2} \leq \frac{n}{2\eta} + \frac{\eta Bt}{2},
    \end{align*}
    where the last inequality holds because $\mathbf{w}^{(i)}, \mathbf{w}^* \in \mathcal{S}_n$. From the convexity of $h$, we have \begin{align*}
        h(\overline{\mathbf{w}}^t) - h(\mathbf{w}^*) \leq \frac{1}{t} \left( \sum_{i=1}^t (h(\mathbf{w}^{(i)}) - h(\mathbf{w}^*)) \right) \leq \frac{n}{2\eta t} + \frac{\eta B}{2}.
    \end{align*}
    By AM-GM inequality, the right hand side is minimized when we choose stepsize $\eta = \sqrt{\frac{n}{Bt}}$, we then obtain \begin{align*}
        h(\overline{\mathbf{w}}^t) - h(\mathbf{w}^*) \leq \sqrt{\frac{nB}{t}}.
    \end{align*}
\end{proof}

We proceed to discuss the convergence rate of Algorithm~\ref{alg:subg}. We define the $\pi$-weighted \textbf{total variation distance} between $Q$ and $P$ as 
\begin{align*}
    D^\pi_\mathrm{TV}(P \| Q) := \dfrac{1}{2}\sum_{x,y \in \mathcal{X}} \pi(x) |P(x,y) - Q(x,y)|,
\end{align*}
and show the convergence rate of Algorithm~\ref{alg:subg}.

\begin{theorem}[Convergence rate of Algorithm~\ref{alg:subg}]\label{thm:subg_TV}
    Consider Algorithm~\ref{alg:subg} and its outputs $(\mathbf{w}^{(i)})_{i=1}^t$, and the stepsize is chosen to be $\eta = \sqrt{\frac{n}{Bt}}$, we have \begin{align*}
        D_\mathrm{TV}^\pi (\otimes_{k=1}^m \overline{P}(\overline{\mathbf{w}})^{(S_k)} \| \otimes_{k=1}^m \overline{P}(\mathbf{w}^*)^{(S_k)}) = \mathcal{O}\left(\frac{1}{\sqrt{t}}\right).
    \end{align*}
\end{theorem}

\begin{proof}
    From the convexity of KL divergence $\D(\cdot \| \cdot)$ and Equation 3.25 of \cite{csiszar1972class}, we have a constant $C$ such that 
    \begin{align*}
        &\quad D_\mathrm{TV}^\pi (\otimes_{k=1}^m \overline{P}(\overline{\mathbf{w}})^{(S_k)} \| \otimes_{k=1}^m \overline{P}(\mathbf{w}^*)^{(S_k)}) \\
        &\leq C \left(\sum_{i=1}^n \overline{w}_i^t \D(P_i \| \otimes_{k=1}^m \overline{P}(\mathbf{w}^*)^{(S_k)}) - \sum_{i=1}^n \overline{w}_i^t \D(P_i \| \otimes_{k=1}^m \overline{P}(\overline{\mathbf{w}}^{(i)})^{(S_k)})\right) \\
        &\leq C \left(\max_{i \in \llbracket n \rrbracket} \D(P_i \| \otimes_{k=1}^m \overline{P}(\mathbf{w}^*)^{(S_k)}) + h(\overline{\mathbf{w}}^t)\right) \\
        &= C(h(\overline{\mathbf{w}}^t) - h(\mathbf{w}^*)) = \mathcal{O} \left(\frac{1}{\sqrt{t}}\right),
    \end{align*}
    where the second last equality comes from the complementary slackness introduced in item \eqref{it:cs} of Theorem~\ref{thm:equiv}, and the last equality comes from Theorem~\ref{thm:subg_ub} as we choose the stepsize $\eta = \sqrt{\frac{n}{Bt}}$.
\end{proof}

\begin{remark}
    Theorem \ref{thm:subg_ub} and Theorem \ref{thm:subg_TV} establish the theoretical guarantee of Algorithm~\ref{alg:subg} through the averaged output $\overline{\mathbf{w}}^t$. However, in numerical experiments, we choose $\argmin_{i \in \llbracket t \rrbracket} h(\mathbf{w}^{(i)})$ as a possible output, see Section \ref{subsec:num_subg}.
\end{remark}

\section{A max-min-max submodular optimization problem and a two-layer subgradient-greedy algorithm}\label{sec:maxminmax}
Recall that in earlier sections we consider the minimax problem \eqref{eq:minmax_KL} and investigate its implications in the two-person game between Nature and probabilist. As the set $\mathcal{F}(\mathbf{S})$ depends on the choice of the partition $\mathbf{S}$, in this section we consider a max-min-max optimization problem of the form
\begin{align*}
    \max_{\mathbf{S} \in (m+1)^{\llbracket d \rrbracket}} \min_{Q \in \mathcal{F}} \max_{\mu \in \mathcal{P}(\mathcal{B})} \int_\mathcal{B} \D(P \| Q) \mu (\mathrm{d} P).
\end{align*}
In words, we seek to find an optimal partition the maximizes the minimal worst-case information loss. We write \begin{align}\label{eq:f_gen_KL}
    f(\mathbf{S}, \mathbf{w}) := \sum_{i=1}^n w_i \D (P_i \| \otimes_{j=1}^m \overline{P}(\mathbf{w})^{(S_j)}),
\end{align} 
and from the mixed-strategy Nash equilibrium (item~\eqref{it:nash_exist} of Theorem~\ref{thm:nash}), we can denote the inner part as 
\begin{align*}
    f(\mathbf{S}, \mathbf{w}^* (\mathbf{S})) &= \min_{Q \in \mathcal{F}} \max_{\mu \in \mathcal{P}(\mathcal{B})} \int_\mathcal{B} \D (P \| Q) \mu(\mathrm{d} P)\\ &= \max_{\mathbf{w} \in \mathcal{S}_n} \sum_{i=1}^n w_i \D (P_i \| \otimes_{j=1}^m \overline{P}(\mathbf{w})^{(S_j)}), \quad \mathbf{S} \in (m+1)^{\llbracket d \rrbracket}\\
    &= \sum_{i=1}^n w^*_i \D(P_i \| \otimes_{j=1}^m \overline{P}(\mathbf{w}^*)^{(S_j)}), \quad \mathbf{S} \in (m+1)^{\llbracket d \rrbracket}\\
    &= \sum_{i=1}^n w^*_i \D (P_i \| (\otimes_{j=1}^{m - 1} \overline{P}(\mathbf{w}^*)^{(S_j)}) \otimes \overline{P}(\mathbf{w}^*)^{(-\mathrm{supp}(\mathbf{S}))}), \quad \mathbf{S} \in m^{\llbracket d \rrbracket},
\end{align*}

where we write $$\mathbf{w}^* = \mathbf{w}^*(\mathbf{S}) = \argmax_{\mathbf{w} \in \mathcal{S}_n} f(\mathbf{S}, \mathbf{w}).$$

We furthermore choose the ground set $\mathbf{V} \in m^{\llbracket d \rrbracket}$ and cardinality constraint $l$, and instead consider the max-min-max optimization problem 
\begin{align}\label{eq:max_min_max}
    \max_{\mathbf{S} \preceq \mathbf{V};~ |\mathrm{supp}(\mathbf{S})| \leq l} f(\mathbf{S}, \mathbf{w}^*(\mathbf{S})).
\end{align}

We then investigate the following map for fixed $\mathbf{w} \in \mathcal{S}_n$ through the lens of submodularity: 
\begin{align}\label{map:SKL_gen}
    m^{\llbracket d \rrbracket} \ni \mathbf{S} \mapsto f(\mathbf{S}) = f(\mathbf{S}, \mathbf{w}) := \sum_{i=1}^n w_i \D (P_i \| (\otimes_{j=1}^{m - 1} \overline{P}(\mathbf{w})^{(S_j)}) \otimes \overline{P}(\mathbf{w})^{(-\mathrm{supp}(\mathbf{S}))}).
\end{align}

\begin{lemma}
    The map \eqref{map:SKL_gen} is orthant submodular.
\end{lemma}

\begin{proof}
    We shall prove that $\Delta_{e, j} f(\mathbf{S}) \geq \Delta_{e, j} f(\mathbf{T})$ from the definition of orthant submodularity, where we choose $\mathbf{S} \preceq \mathbf{T}$ and $e \notin \mathrm{supp}(\mathbf{T})$.
    \begin{align*}
        \Delta_{e, j} f(\mathbf{S}) - \Delta_{e, j} f(\mathbf{T}) &= \sum_{i=1}^n w_i \left(H(\overline{P}^{(S_j \cup \{e\})}) - H(\overline{P}^{(S_j)}) + H(\overline{P}^{(-\mathrm{supp}(\mathbf{S}) \cup \{e\})}) - H(\overline{P}^{(-\mathrm{supp}(\mathbf{S}))})\right) \\
        &\quad - \sum_{i=1}^n w_i \left(H(\overline{P}^{(T_j \cup \{e\})}) - H(\overline{P}^{(T_j)}) + H(\overline{P}^{(-\mathrm{supp}(\mathbf{T}) \cup \{e\})}) - H(\overline{P}^{(-\mathrm{supp}(\mathbf{T}))})\right) \\
        &= \left[\left(H(\overline{P}^{(S_j \cup \{e\})}) - H(\overline{P}^{(S_j)})\right) - \left(H(\overline{P}^{(T_j \cup \{e\})}) - H(\overline{P}^{(T_j)})\right)\right] \\
        &\quad + \left[\left(H(\overline{P}^{(-\mathrm{supp}(\mathbf{T}))}) - H(\overline{P}^{(-\mathrm{supp}(\mathbf{T})\cup \{e\})})\right) - \left(H(\overline{P}^{(-\mathrm{supp}(\mathbf{S}))}) - H(\overline{P}^{(-\mathrm{supp}(\mathbf{S}) \cup \{e\})})\right)\right].
    \end{align*}
    Since the map $S \mapsto H(\overline{P}^{(S)})$ is submodular (see item~\ref{it:submod_ent_w} of Theorem \ref{thm:submod_mc}) and $\mathbf{S} \preceq \mathbf{T}$, then we have \begin{align*}
        \left(H(\overline{P}^{(S_j \cup \{e\})}) - H(\overline{P}^{(S_j)})\right) - \left(H(\overline{P}^{(T_j \cup \{e\})}) - H(\overline{P}^{(T_j)})\right) &\geq 0,\\
        \left(H(\overline{P}^{(-\mathrm{supp}(\mathbf{T}))}) - H(\overline{P}^{(-\mathrm{supp}(\mathbf{T})\cup \{e\})})\right) - \left(H(\overline{P}^{(-\mathrm{supp}(\mathbf{S}))}) - H(\overline{P}^{(-\mathrm{supp}(\mathbf{S}) \cup \{e\})})\right) &\geq 0.
    \end{align*}
    Therefore $\Delta_{e, j} f(\mathbf{S}) - \Delta_{e, j} f(\mathbf{T}) \geq 0$ and hence the map \eqref{map:SKL_gen} is orthant submodular.
\end{proof}

In view of Theorem 2.6 of \cite{lai2025information}, since the map~\eqref{map:SKL_gen} is orthant submodular, then for any $\beta = \beta(\mathbf{w}) \in \mathbb{R}$, if $\mathbf{S} \preceq \mathbf{V}$, we have the following monotonically non-decreasing $(m - 1)$-submodular function:
\begin{align}\label{eq:g_KL_gen}
    g(\mathbf{S}, \mathbf{w}) &:= f(\mathbf{S}) - \beta + \sum_{j=1}^{m - 1} \sum_{e \in S_j} (f(V_1, \ldots, V_j, \ldots, V_{m - 1})) - f(V_1, \ldots, V_j \backslash \{e\}, \ldots, V_{m - 1})) \nonumber\\
    &= f(\mathbf{S}) - \beta + \sum_{i=1}^n \sum_{j=1}^{m - 1} \sum_{e \in S_j} w_i \left[\D(\overline{P}^{(V_j)} \| \overline{P}^{(V_j \backslash \{e\})} \otimes \overline{P}^{(e)}) - \D(\overline{P}^{(-\mathrm{supp}(\mathbf{V}) \backslash \{e\})} \| \overline{P}^{(-\mathrm{supp}(\mathbf{V}))} \otimes \overline{P}^{(e)})\right] \nonumber\\
    &= f(\mathbf{S}) - \beta + \sum_{j=1}^{m - 1} \sum_{e \in S_j} \left[\D(\overline{P}^{(V_j)} \| \overline{P}^{(V_j \backslash \{e\})} \otimes \overline{P}^{(e)}) - \D(\overline{P}^{(-\mathrm{supp}(\mathbf{V}) \backslash \{e\})} \| \overline{P}^{(-\mathrm{supp}(\mathbf{V}))} \otimes \overline{P}^{(e)})\right],
\end{align}
where the last equality comes from the fact that $\mathbf{w} \in \mathcal{S}_n$.

We also obtain the following modular function:
\begin{align}\label{eq:c_KL_gen}
    c(\mathbf{S}, \mathbf{w}) = -\beta + \sum_{j=1}^{m - 1} \sum_{e \in S_j} \left[\D(\overline{P}^{(V_j)} \| \overline{P}^{(V_j \backslash \{e\})} \otimes \overline{P}^{(e)}) - \D(\overline{P}^{(-\mathrm{supp}(\mathbf{V}) \backslash \{e\})} \| \overline{P}^{(-\mathrm{supp}(\mathbf{V}))} \otimes \overline{P}^{(e)})\right],
\end{align}
where we take 
\begin{align}\label{eq:beta_KL_gen}
    \beta = \beta(\mathbf{w}) \leq - \sum_{j=1}^{m - 1} \sum_{e \in S_j} \left[H(\overline{P}(\mathbf{w})^{(-\mathrm{supp}(\mathbf{V}) \cup \{e\})}) + H(\overline{P}(\mathbf{w})^{(e)})\right]
\end{align}
and write $c(\mathbf{S}, \mathbf{w}) \leq C$ to ensure that $0 \leq c \leq C$. Therefore, for fixed $\mathbf{w} \in \mathcal{S}_n$,
\begin{align*}
    f(\mathbf{S}, \mathbf{w}) = g(\mathbf{S}, \mathbf{w}) - c(\mathbf{S}, \mathbf{w}),
\end{align*}
where $f$ can be written as the difference between a $(m-1)$-submodular function and a non-negative modular function.

\begin{remark}
    If we consider the optimization problem \eqref{eq:max_min_max} with fixed $\mathbf{w} \in \mathcal{S}_n$, i.e., \begin{align*}
        \max_{\mathbf{S} \preceq \mathbf{V};~ |\mathrm{supp}(\mathbf{S})| \leq l} f(\mathbf{S}) = f(\mathbf{S}, \mathbf{w}),
    \end{align*}
    we can apply Algorithm 3 of \cite{lai2025information} with $g$ as in \eqref{eq:g_KL_gen}, $c$ as in \eqref{eq:c_KL_gen}, and $\beta$ as in \eqref{eq:beta_KL_gen} to solve the problem. Furthermore, Theorem 2.11 of \cite{lai2025information} gives the following lower bound: \begin{align*}
        f(\mathbf{S}_l, \mathbf{w}) \geq (1 - e^{-1})g(\mathbf{OPT}, \mathbf{w}) - c(\mathbf{OPT}, \mathbf{w}),
    \end{align*}
    where $\mathbf{S}_l = (S_{l, 1}, \ldots, S_{l, m-1})$ is the final output of Algorithm 3 of \cite{lai2025information} and $\mathbf{OPT} = \argmax_{\mathbf{S} \preceq \mathbf{V};~ |\mathrm{supp}(\mathbf{S})| \leq l} f(\mathbf{S})$.
\end{remark}

\SetKwInput{Input}{Input}
\SetKwInput{Output}{Output}

We propose Algorithm~\ref{alg:max_max} to solve problem \eqref{eq:max_min_max}. Algorithm~\ref{alg:max_max} is a two-layer subgradient-greedy algorithm, which combines the outer generalized distorted greedy algorithm (Algorithm 3 of \cite{lai2025information}) and the inner projected subgradient algorithm (Algorithm~\ref{alg:subg}). Specifically, we conduct totally $l$ rounds of generalized distorted greedy algorithm: at the $i$-th round, we first fix $\mathbf{S}_i$ and apply the projected subgradient algorithm on fixed $\mathbf{S}_i$ for $K$ iterations to maximize the objective function $f(\mathbf{S}_i, \cdot)$; we then fix $\overline{\mathbf{w}}_{i+1} = \sum_{k=1}^K \mathbf{w}_{i+1}^{(k)}$ and perform generalized distorted greedy algorithm to obtain $\mathbf{S}_{i+1}$. We proceed to state and prove a lower bound of Algorithm~\ref{alg:max_max} in Theorem~\ref{thm:max_max_lb}.

\begin{algorithm}[h]
\caption{A two-layer subgradient-greedy algorithm to solve problem \eqref{eq:max_min_max}}
\label{alg:max_max}
\Input{$f$ as in \eqref{eq:f_gen_KL}; $g$ as in \eqref{eq:g_KL_gen}; $c$ as in \eqref{eq:c_KL_gen}; subgradient $\mathbf{g}$ as in Theorem~\ref{thm:subg}; cardinality constraint $l$; partition of ground set $\mathbf{V}=(V_1,\ldots,V_{m - 1})\in m^{\llbracket d \rrbracket}$; inner iteration number $K$.}
\Output{Coordinates $\mathbf{S}_l=(S_{l, 1}, \ldots, S_{l, m - 1})$ and weights $\overline{\mathbf{w}}^{(l)}$.}
Initialize $\mathbf{S}_0=(S_{0,1},\ldots,S_{0,m - 1}) \gets \emptyset$ and $\mathbf{w}^{(K)}_0 = (\frac{1}{m}, \ldots, \frac{1}{m})$. \\
Compute bound $B$ as in Theorem~\ref{thm:subg} and stepsize $\eta = \sqrt{\frac{n}{BK}}$.\\
\For{$i=0$ \KwTo $l-1$}{
  $\mathbf{w}^{(0)}_{i + 1} \gets \mathbf{w}^{(K)}_i$.
  
\For{$k=0$ \KwTo $K-1$}{
  $\mathbf{v} \gets \mathbf{w}^{(k)}_{i + 1} - \eta \cdot \mathbf{g}(\mathbf{S}_i, \mathbf{w}^{(k)}_{i + 1})$. 
  
  $\mathbf{w}^{(k + 1)}_{i + 1} \gets \argmin_{\mathbf{w} \in \mathcal{S}_n} \|\mathbf{w} - \mathbf{v}\|_2^2$.
}  
  $\overline{\mathbf{w}}_{i + 1} \gets \frac{1}{K} \sum_{k=1}^K \mathbf{w}^{(k)}_{i + 1}$.
  
  $(j^*,e^*) \gets \argmax\limits_{\substack{j \in \llbracket m - 1 \rrbracket;\\ e \in V_j \setminus S_{i,j}}}
  \left\{\left(1-\frac{1}{l}\right)^{l-(i+1)} \Delta_{e,j} g(\mathbf{S}_i, \overline{\mathbf{w}}_{i + 1}) - c(\{e\}, \overline{\mathbf{w}}_{i + 1})\right\}$. 
  
  \eIf{$\left(1-\frac{1}{l}\right)^{l-(i+1)} \Delta_{e^*,j^*} g(\mathbf{S}_i, \overline{\mathbf{w}}_{i + 1}) - c(\{e^*\}, \overline{\mathbf{w}}_{i + 1}) > 0$}{
    $S_{i+1,j^*} \gets S_{i,j^*} \cup \{e^*\}$.
  }{
    $S_{i+1,j^*} \gets S_{i,j^*}$.
  }
  \For{$ k \in \llbracket m - 1 \rrbracket,\; k\neq j^*$}{
    $S_{i+1, k} \gets S_{i, k}$.
  }
}
\Return $\mathbf{S}_l$ and $\overline{\mathbf{w}}_l$.
\end{algorithm}

\begin{theorem}[Lower bound of Algorithm~\ref{alg:max_max}]\label{thm:max_max_lb}
    Algorithm~\ref{alg:max_max} provides the following lower bound:
    \begin{align*}
        f(\mathbf{S}_l, \overline{\mathbf{w}}_l) > \frac{1}{l} \sum_{i=1}^l [\alpha_i g(\mathbf{OPT}(\overline{\mathbf{w}}_i), \overline{\mathbf{w}}_i) - c(\mathbf{OPT}(\overline{\mathbf{w}}_i), \overline{\mathbf{w}}_i)] - \mathcal{O} \left(l\left(\sqrt{\frac{nB}{K}} + C\right)\right),
    \end{align*}
    where $(\mathbf{S}_l, \overline{\mathbf{w}}_l)$ is the output of Algorithm~\ref{alg:max_max}, $\alpha_i = (1 - \frac{1}{l})^{l - i}$, and \begin{align*}
        \mathbf{OPT}(\mathbf{w}) = \argmax_{\mathbf{S} \preceq \mathbf{V};~ |\mathrm{supp}(\mathbf{S})| \leq l} f(\mathbf{S}, \mathbf{w}).
    \end{align*}
\end{theorem}

\begin{proof}
We define the distorted objective function $\Phi_i: m^{\llbracket d \rrbracket} \times \mathcal{S}_n \to \mathbb{R}$ to be \begin{align*}
    \Phi_i(\mathbf{S}, \overline{\mathbf{w}}_i) := \alpha_i g(\mathbf{S}, \overline{\mathbf{w}}_i) - c(\mathbf{S}, \overline{\mathbf{w}}_i) > \alpha_i f(\mathbf{S}, \overline{\mathbf{w}}_i) - c(\mathbf{S}, \overline{\mathbf{w}}_i),
\end{align*}
where the inequality comes from the fact that $0 < \alpha_i \leq 1$.

We look into the difference of the distorted objective function
\begin{align*}
    \Phi_{i+1} (\mathbf{S}_{i + 1}, \overline{\mathbf{w}}_{i+1}) - \Phi_i(\mathbf{S}_i, \overline{\mathbf{w}}_{i}) = [\Phi_{i + 1}(\mathbf{S}_{i+1}, \overline{\mathbf{w}}_{i+1}) - \Phi_i (\mathbf{S}_i, \overline{\mathbf{w}}_{i+1})] - [\Phi_i (\mathbf{S}_i, \overline{\mathbf{w}}_{i+1}) - \Phi_i(\mathbf{S}_i, \overline{\mathbf{w}}_{i})],
\end{align*}
where the first term is the gain in the distorted greedy algorithm, and the second term is the weight update error.

We first refer to the proof of Theorem 2.11 of \cite{lai2025information} and state the lower bound of the gain in the distorted greedy algorithm \begin{align*}
    \Phi_{i + 1}(\mathbf{S}_{i+1}, \overline{\mathbf{w}}_{i+1}) - \Phi_i (\mathbf{S}_i, \overline{\mathbf{w}}_{i+1}) \geq \frac{1}{l} (\alpha_{i+1}g(\mathbf{OPT}(\overline{\mathbf{w}}_{i+1}), \overline{\mathbf{w}}_{i+1}) - c(\mathbf{OPT}(\overline{\mathbf{w}}_{i+1}), \overline{\mathbf{w}}_{i+1})).
\end{align*}
We then analyze the weight update error term. From Theorem~\ref{thm:subg_ub}, we have \begin{align*}
    f(\mathbf{S}_i, \mathbf{w}^*(\mathbf{S}_i)) - f(\mathbf{S}_i, \overline{\mathbf{w}}_m) \leq \sqrt{\frac{nB}{K}},~ \forall m\in \llbracket l \rrbracket.
\end{align*}
hence the lower bound of the weight update error is  \begin{align*}
    \Phi_i (\mathbf{S}_i, \overline{\mathbf{w}}_{i+1}) - \Phi_i(\mathbf{S}_i, \overline{\mathbf{w}}_{i}) &= \alpha_i(f(\mathbf{S}_i, \overline{\mathbf{w}}_{i+1}) - f(\mathbf{S}_i, \overline{\mathbf{w}}_i)) - (c(\mathbf{S}_i, \overline{\mathbf{w}}_{i+1}) - c(\mathbf{S}_i, \overline{\mathbf{w}}_i)) \\
    &> -\alpha_i \|f(\mathbf{S}_i, \overline{\mathbf{w}}_{i+1}) - f(\mathbf{S}_i, \overline{\mathbf{w}}_i)\| - C\\
    &\geq -\alpha_i (\|f(\mathbf{S}_i, \mathbf{w}^*(\mathbf{S}_i)) - f(\mathbf{S}_i, \overline{\mathbf{w}}_{i+1})\| + \|f(\mathbf{S}_i, \mathbf{w}^*(\mathbf{S}_i)) - f(\mathbf{S}_i, \overline{\mathbf{w}}_i)\|) - C\\
    &\geq -2\alpha_i \sqrt{\frac{nB}{K}} - C.
\end{align*}

Since $\Phi_0(\mathbf{S}_0) \geq 0$, then $$f(\mathbf{S}_l, \overline{\mathbf{w}}_l) = \alpha_l \cdot g(\mathbf{S}_l, \overline{\mathbf{w}}_i) - c(\mathbf{S}_l, \overline{\mathbf{w}}_i) \geq \sum_{i=0}^{l - 1} [\Phi_{i + 1} (\mathbf{S}_{i + 1}) - \Phi_i (\mathbf{S}_i)], $$ hence
\begin{align*}
    f(\mathbf{S}_l, \overline{\mathbf{w}}_l) &\geq \sum_{i=0}^{l-1} [\Phi_{i + 1}(\mathbf{S}_{i+1}, \overline{\mathbf{w}}_{i+1}) - \Phi_i (\mathbf{S}_i, \overline{\mathbf{w}}_{i+1})] + \sum_{i=0}^{l-1} [\Phi_i (\mathbf{S}_i, \overline{\mathbf{w}}_{i+1}) - \Phi_i(\mathbf{S}_i, \overline{\mathbf{w}}_{i})] \\
    &> \frac{1}{l} \sum_{i=1}^l [\alpha_i g(\mathbf{OPT}(\overline{\mathbf{w}}_i), \overline{\mathbf{w}}_i) - c(\mathbf{OPT}(\overline{\mathbf{w}}_i), \overline{\mathbf{w}}_i)] - 2\sqrt{\frac{nB}{K}} \sum_{i=0}^{l-1} \alpha_i - lC\\
    &= \frac{1}{l} \sum_{i=1}^l [\alpha_i g(\mathbf{OPT}(\overline{\mathbf{w}}_i), \overline{\mathbf{w}}_i) - c(\mathbf{OPT}(\overline{\mathbf{w}}_i), \overline{\mathbf{w}}_i)] - \mathcal{O} \left(l\left(\sqrt{\frac{nB}{K}} + C\right)\right).
\end{align*}
\end{proof}

\section[Numerical Experiments]{Numerical experiments\footnote{The code is available at: \href{https://github.com/zheyuanlai/subgradient-greedy/}{\texttt{https://github.com/zheyuanlai/subgradient-greedy}}.}}\label{sec:num}
We conduct a series of numerical experiments to validate the theoretical framework and evaluate the performance of the proposed algorithms. The experiments are designed to demonstrate the performance of the projected subgradient algorithm (Algorithm~\ref{alg:subg}) to solve problem \eqref{eq:subg} and the two-layer subgradient-greedy algorithm (Algorithm~\ref{alg:max_max}) to solve problem \eqref{eq:max_min_max} on the multivariate Markov chains associated with the Curie-Weiss model and the Bernoulli-Laplace level model.

\subsection{Experiment settings}
\subsubsection{Curie-Weiss model}
We aim to generate a $d$-dimensional Markov chain from the Curie-Weiss model. We consider a discrete $d$-dimensional hypercube state space given by
\begin{align*}
    \mathcal{X} = \{-1,+1\}^{d}.
\end{align*}
Let the Hamiltonian function be that of the Curie-Weiss model (see Chapter 13 of~\cite{bovier2016metastability}) on $\mathcal{X}$ with interaction coefficients $\frac{1}{2^{|j-i|}}$ and external magnetic field $h=1$, that is, for $x = (x^1, \ldots, x^{d}) \in \mathcal{X}$,
\begin{align*}
    \mathcal{H}(x) = - \sum_{i=1}^{d} \sum_{j=1}^{d} \dfrac{1}{2^{|j-i|}} x^i x^j - h \sum_{i=1}^d x^i.
\end{align*}
We consider a Glauber dynamics with a simple random walk proposal targeting the Gibbs distribution at temperature $T=10$. At each step we pick uniformly at random one of the $d$ coordinates and flip it to the opposite sign, along with an acceptance-rejection filter, that is,
\begin{align*}
    P(x, y) = \begin{cases}
\dfrac{1}{d} e^{-\frac{1}{T} (\mathcal{H}(y) - \mathcal{H}(x))_+}, & \text{if } y = (x^1,x^2,\ldots,-x^i,\ldots,x^d), i \in \llbracket d \rrbracket, \\
1 - \sum_{y;~y \neq x} P(x, y), & \text{if } x = y, \\
0, & \text{otherwise},
\end{cases}
\end{align*}
where for $m \in \mathbb{R}$ we denote $m_+ := \max\{m,0\}$ as the non-negative part of $m$. The stationary distribution of $P$ is the Gibbs distribution at temperature $T$ given by
\begin{align*}
    \pi(x) = \dfrac{e^{-\frac{1}{T} \mathcal{H}(x)}}{\sum_{z \in \mathcal{X}} e^{-\frac{1}{T} \mathcal{H}(z)}}.
\end{align*}

\subsubsection{Bernoulli-Laplace level model}
We aim to generate a $d$-dimensional Markov chain from the Bernoulli-Laplace level model. We consider a $(d+1)$-dimensional Bernoulli–Laplace level model as described in Section 4.2 of~\cite{khare2009rates}. Let \begin{align*}
    \mathcal{X} = \{x = (x^1, \ldots, x^{d+1}) \in \mathbb{N}_0^{d+1};~ x^1 + \ldots + x^{d+1} = N\}
\end{align*}
be the state space, where $x^i$ can be interpreted as the number of ``particles'' of type $i$ out of the total number $N=d$. The stationary distribution of such Markov chain, $\pi$, is given by the multivariate hypergeometric distribution described in Lemma 4.18 of~\cite{khare2009rates}. Concretely, we have
\begin{align*}
    \pi(x) = \frac{\prod_{i=1}^{d+1} {l_i \choose x^i}}{{l_1 + \ldots + l_{d+1} \choose N}}, \quad x \in \mathcal{X}, 
\end{align*}
for some fixed parameters $l_1 = \ldots = l_{d} = 1$ and $l_{d+1} = d$, which represents the total number of ``particles'' of type $i$. Under this setting, we let $x^{d+1} = N - \sum_{i=1}^{d} x^i$, and hence the state space is of product form with $\mathcal{X} = \{0, 1\}^{d}$.

Following the spectral decomposition for reversible Markov chains (see Section 2.1 of \cite{khare2009rates} for background), the transition matrix $P$ is written as:
\begin{align*}
    P(x, y) = \sum_{n=0}^N \beta_n \phi_n(x) \phi_n (y) \pi(y),
\end{align*}
where $\beta_n$ are the eigenvalues and $\phi_n(x)$ is the associated eigenfunction.

From Definition 4.15 of~\cite{khare2009rates}, in the Bernoulli-Laplace level model, we choose $s=1$ as the swap size parameter satisfying $$0 \leq s \leq \min \left\{N, \sum_{i=1}^{d+1} l_i - N\right\},$$ where we consider $\sum_{i=1}^{d+1} l_i > N$. From Theorem 4.19 of~\cite{khare2009rates}, the eigenvalues for the Bernoulli-Laplace level model are given by
\begin{align*}
    \beta_n = \sum_{k=0}^n {n \choose k} \frac{(N - s)_{[n-k]} s_{[k]}}{N_{[n-k]} \left(\sum_{i=1}^{d+1} l_i - N\right)_{[k]}}, \quad 0\leq n \leq N,
\end{align*}
where $a_{[k]} = a (a-1) \cdots (a - k + 1)$, and we apply the convention that $a_{[0]} = 1$. 

In this case, we choose the eigenfunction as \begin{align*}
    \phi_n (x) = \left\{\mathbf{Q_n}\left(x; N, -\sum_{i=1}^{d+1} l_i\right)\right\}_{|\mathbf{n}| = n},
\end{align*}
where $\mathbf{Q_n}$ are the multivariate Hahn polynomials for the hypergeometric distribution as defined in Proposition 2.3 of~\cite{khare2009rates}.

\subsection{Numerical experiments of Algorithm~\ref{alg:subg}}\label{subsec:num_subg}
We apply the projected subgradient algorithm (Algorithm~\ref{alg:subg}) to solve problem \eqref{eq:subg} for both the Curie-Weiss and Bernoulli-Laplace level models. We start with a low-dimensional example. For both settings, we construct a 5-dimensional Markov chain with $\pi$-stationary transition probability matrix $P$ on state space $\mathcal{X} = \{0, 1\}^5$. We then construct a family of $n=5$ transition matrices with $\mathcal{B} = \{P, P^2, P^4, P^8, P^{16}\}$, which ensures that all matrices in $\mathcal{B}$ share the same stationary distribution $\pi$. We partition the state space into $\mathbf{S} = \{S_1, S_2, S_3\}$ ($m=3$) such that $S_1=\{1,2\}$, $S_2=\{3,5\}$, and $S_3=\{4\}$.

We initialize the algorithm with uniform weights $\mathbf{w}^{(0)} = (1/5, \ldots, 1/5)$. The step size is chosen according to the theoretical guarantee from Theorem~\ref{thm:subg_ub}, $\eta = \sqrt{\frac{n}{Bt}}$, where the subgradient norm bound $B$ is estimated once at the beginning of the algorithm. The number of iterations until convergence is theoretically determined by $t = \lceil \frac{nB}{\epsilon^2}\rceil$, but $t$ would be large with large $B$ and small $\epsilon$. Therefore for practical purpose, we only run a small number of iterations for demonstration. The trajectory plots of the projected subgradient algorithm and the evolution of weights of both models are shown in Figure~\ref{fig:psg_5}. We also summarize the weights $\mathbf{w} \in \mathcal{S}_n$ and the corresponding objective value $h(\mathbf{w})$ in Table~\ref{tab:h_5} for both Curie-Weiss and Bernoulli-Laplace models. We state and compare the optimal $\mathbf{w}$ during the optimization process $\argmin_{i \in \llbracket t \rrbracket} h(\mathbf{w}^{(i)})$, the averaged value during the iterations $\overline{\mathbf{w}}^t$, initial uniform $\mathbf{w}^{(0)}$, extreme weight $\mathbf{w}_\mathrm{ex}$ such that only $\mathbf{w}_{\mathrm{ex}, 0} = 1$, and the final weight $\mathbf{w}^{(t)}$ of the iterations.

\begin{figure}[H]
    \centering
    \begin{subfigure}[b]{0.48\textwidth}
        \includegraphics[width=\textwidth]{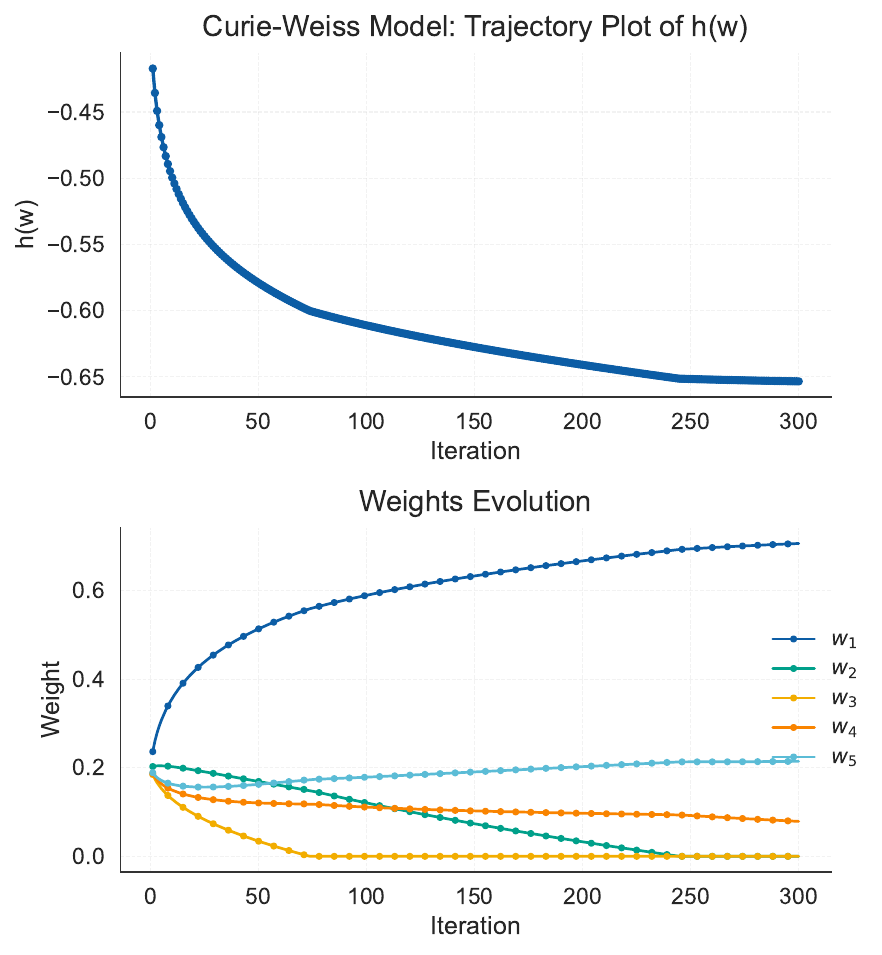}
        \caption{Curie-Weiss model}
        \label{fig:psg_cw_5}
    \end{subfigure}
    \hfill
    \begin{subfigure}[b]{0.48\textwidth}
        \includegraphics[width=\textwidth]{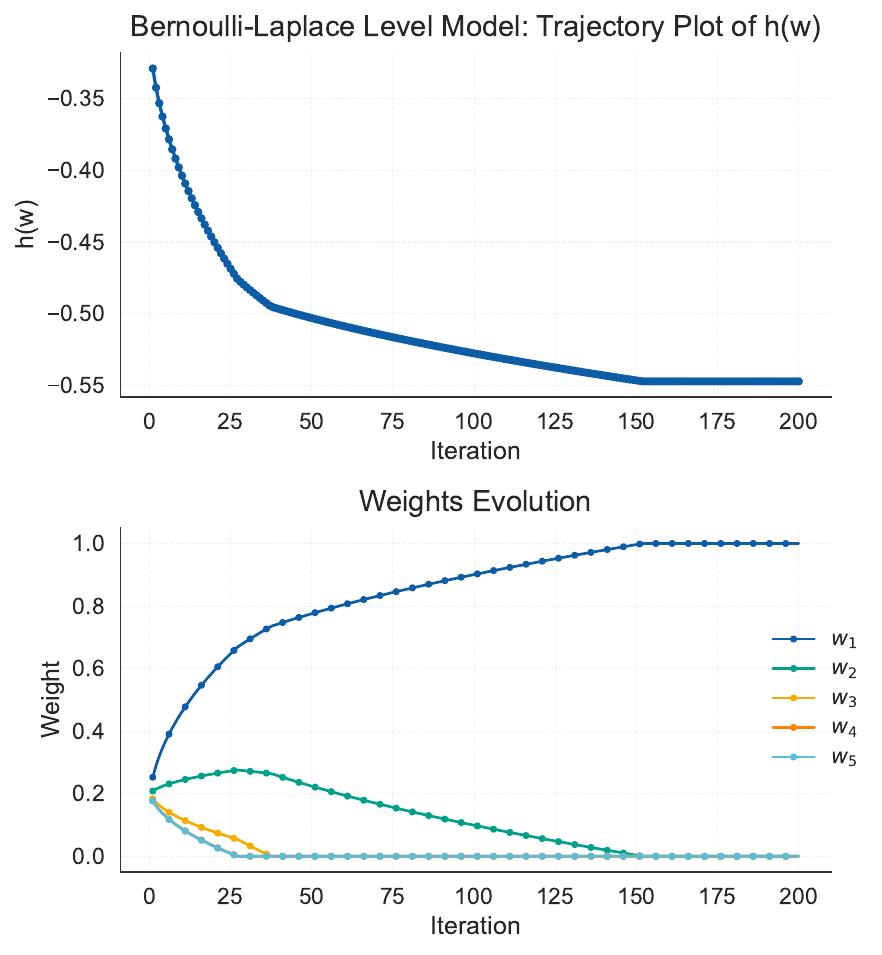}
        \caption{Bernoulli-Laplace level model}
        \label{fig:psg_bl_5}
    \end{subfigure}
    \caption{Convergence of the projected subgradient algorithm for both models ($d=5$).}
    \label{fig:psg_5}
\end{figure}

\begin{table}[h]
\centering
\small
\begin{tabular}{l|cc}
\hline
$\mathbf{w}$, $h(\mathbf{w})$ / \textbf{Model} & \textbf{Curie-Weiss} & \textbf{Bernoulli-Laplace} \\
\hline
$\argmin_{i \in \llbracket t \rrbracket} h(\mathbf{w}^{(i)})$ & $(0.71, 0.00, 0.00, 0.08, 0.21)$ & $(1.00, 0.00, 0.00, 0.00, 0.00)$ \\
$\overline{\mathbf{w}}^t$ & $(0.60, 0.08, 0.02, 0.11, 0.19)$ & $(0.85, 0.11, 0.02, 0.01, 0.01)$ \\
$\mathbf{w}^{(0)}$ & $(0.20, 0.20, 0.20, 0.20, 0.20)$ & $(0.20, 0.20, 0.20, 0.20, 0.20)$ \\
$\mathbf{w}_\mathrm{ex}$ & $(1.00, 0.00, 0.00, 0.00, 0.00)$ & $(1.00, 0.00, 0.00, 0.00, 0.00)$ \\
$\mathbf{w}^{(t)}$ & $(0.71, 0.00, 0.00, 0.08, 0.21)$ & $(1.00, 0.00, 0.00, 0.00, 0.00)$ \\
\hline
$\min_{i \in \llbracket t \rrbracket} h(\mathbf{w}^{(i)})$ & $-0.65$ & $-0.55$ \\
$h(\overline{\mathbf{w}}^t)$ & $-0.62$ & $-0.51$ \\
$h(\mathbf{w}^{(0)})$ & $-0.39$ & $-0.31$ \\
$h(\mathbf{w}_\mathrm{ex})$ & $-0.48$ & $-0.55$ \\
$h(\mathbf{w}^{(t)})$ & $-0.65$ & $-0.55$ \\
\hline
\end{tabular}
\caption{Comparison of $h(\mathbf{w})$ values for different weight choices ($d=5$)}\label{tab:h_5}
\end{table}

For the Curie-Weiss model (Figure~\ref{fig:psg_cw_5}), the algorithm demonstrates rapid initial decrease, after the first 50 iterations, the objective value decreases with a slower rate, which totally converges after 250 iterations. The weights converge to a sparse distribution, with the final weight vector being approximately $\mathbf{w}^{(t)} = (0.71, 0.00, 0.00, 0.08, 0.21)$. This indicates that the final solution is approximately a convex combination of the base transition matrix $P$ and the transition matrix with the highest mixing rate $P^{16}$, while the intermediate transition matrices have zero weights.

The Bernoulli-Laplace level model (Figure~\ref{fig:psg_bl_5}) exhibits similar convergence behavior: the objective value decreases fast in the first 30 steps, then it moves slowly until fully converged after 150 iterations. The final weight vector converges to $\mathbf{w}^{(t)} =(1.00, 0.00, 0.00, 0.00, 0.00)$, indicating that the optimal solution is entirely the base transition matrix $P$.

We then conduct experiments associated with the family of transition matrices including lazy Markov chain (see e.g. \cite{shen2014lazy} for background). Precisely, we choose 

\begin{align*}
    \mathcal{B} = \left\{P, P^2, P^4, \frac{1}{4} I + \frac{3}{4} P, \frac{1}{2} (I+P), \frac{3}{4} I + \frac{1}{4} P \right\},
\end{align*}
where one readily verifies that all the transition matrices in family $\mathcal{B}$ share the same stationary distribution $\pi$. The trajectory plots are shown in Figure~\ref{fig:psg_lazy}, and we also summarize the objective values of different $\mathbf{w}$'s in Table~\ref{tab:h_5_lazy}.

\begin{figure}[H]
    \centering
    \begin{subfigure}[b]{0.48\textwidth}
        \includegraphics[width=\textwidth]{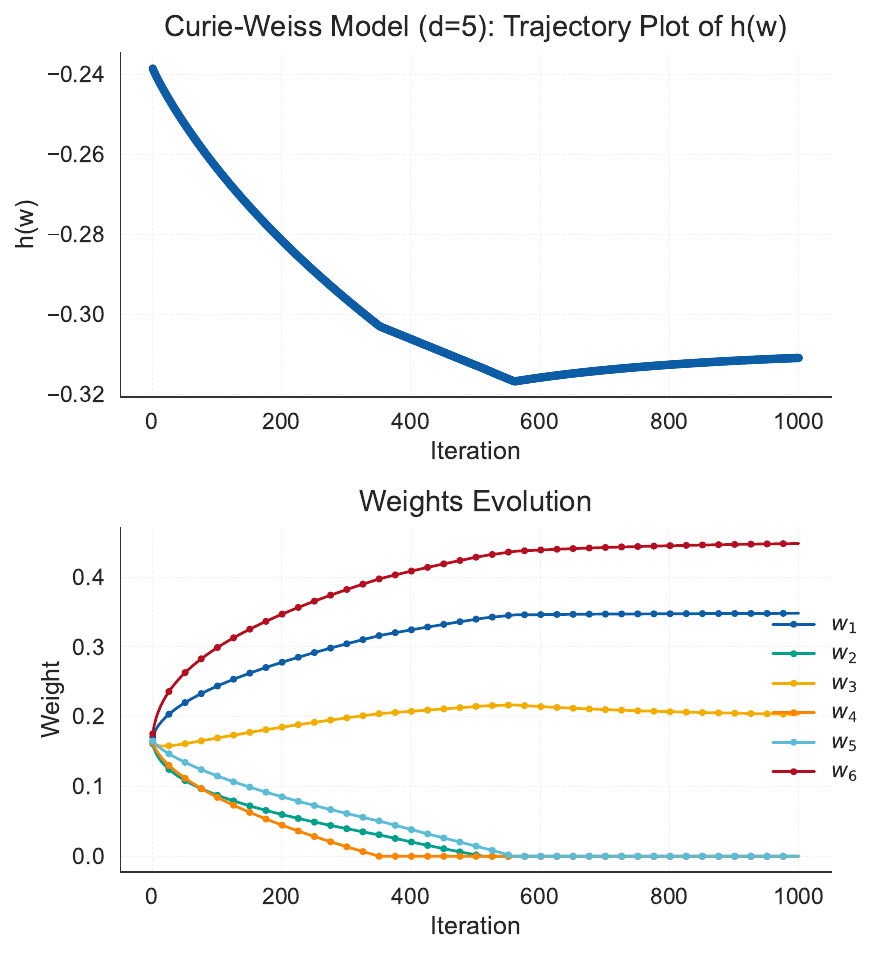}
        \caption{Curie-Weiss model}
        \label{fig:psg_cw_lazy}
    \end{subfigure}
    \hfill
    \begin{subfigure}[b]{0.48\textwidth}
        \includegraphics[width=\textwidth]{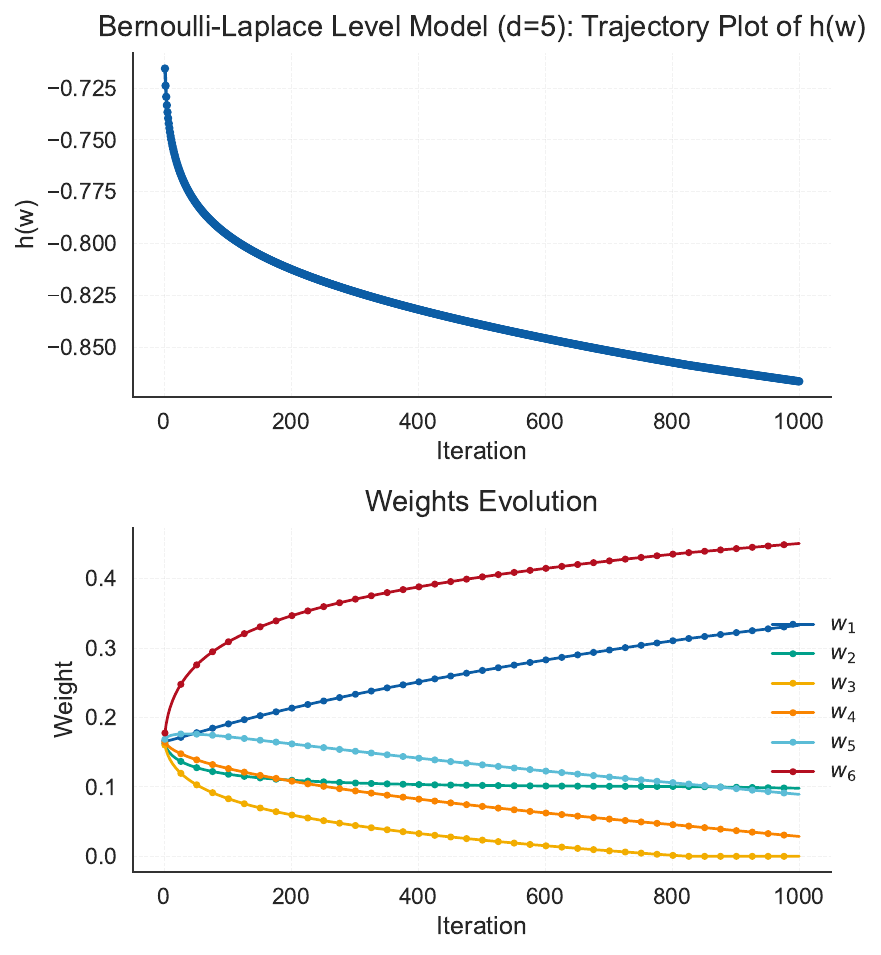}
        \caption{Bernoulli-Laplace level model}
        \label{fig:psg_bl_lazy}
    \end{subfigure}
    \caption{Trajectory plot of the projected subgradient algorithm for both models (incl. lazy chains).}
    \label{fig:psg_lazy}
\end{figure}

\begin{table}[h]
\centering
\small
\begin{tabular}{l|cc}
\hline
$\mathbf{w}$, $h(\mathbf{w})$ / \textbf{Model} & \textbf{Curie-Weiss} & \textbf{Bernoulli-Laplace} \\
\hline
$\argmin_{i \in \llbracket t \rrbracket} h(\mathbf{w}^{(i)})$ & $(0.35, 0.00, 0.22, 0.00, 0.00, 0.44)$ & $(0.33, 0.10, 0.00, 0.03, 0.09, 0.45)$ \\
$\overline{\mathbf{w}}^t$ & $(0.32, 0.03, 0.20, 0.02, 0.04, 0.40)$ & $(0.26, 0.11, 0.03, 0.08, 0.13, 0.39)$ \\
$\mathbf{w}^{(0)}$ & $(0.17, 0.17, 0.17, 0.17, 0.17, 0.17)$ & $(0.17, 0.17, 0.17, 0.17, 0.17, 0.17)$ \\
$\mathbf{w}_\mathrm{ex}$ & $(1.00, 0.00, 0.00, 0.00, 0.00, 0.00)$ & $(1.00, 0.00, 0.00, 0.00, 0.00, 0.00)$ \\
$\mathbf{w}^{(t)}$ & $(0.35, 0.00, 0.20, 0.00, 0.00, 0.45)$ & $(0.33, 0.10, 0.00, 0.03, 0.09, 0.45)$ \\
\hline
$\min_{i \in \llbracket t \rrbracket} h(\mathbf{w}^{(i)})$ & $-0.32$ & $-0.87$ \\
$h(\overline{\mathbf{w}}^t)$ & $-0.34$ & $-0.31$ \\
$h(\mathbf{w}^{(0)})$ & $-0.28$ & $-0.29$ \\
$h(\mathbf{w}_\mathrm{ex})$ & $-0.29$ & $-0.55$ \\
$h(\mathbf{w}^{(t)})$ & $-0.31$ & $-0.87$ \\
\hline
\end{tabular}
\caption{Comparison of $h(\mathbf{w})$ values for different weight choices (incl. lazy chains)}
\label{tab:h_5_lazy}
\end{table}

For the Curie-Weiss model (Figure~\ref{fig:psg_cw_lazy}), the algorithm exhibits an initial decrease followed by a slight increase towards convergence. Since the projected subgradient algorithm (Algorithm~\ref{alg:subg}) is not a descent algorithm, then it is not guaranteed that $h$ shows a non-decreasing trajectory. The final objective value reaches approximately $-0.311$, while the final weight learned by the algorithm is
\[
\mathbf{w}^{(t)} = \Big(
\underbrace{0.35}_{P},\;
\underbrace{0.00}_{P^2},\;
\underbrace{0.20}_{P^4},\;
\underbrace{0.00}_{\frac{1}{4}I+\frac{3}{4}P},\;
\underbrace{0.00}_{\frac{1}{2}(I+P)},\;
\underbrace{0.45}_{\frac{3}{4}I+\frac{1}{4}P}
\Big),
\]
which is sparse and concentrates on three extremes: the base chain $P$, the most accelerated $P^4$, and the ``laziest'' member $\tfrac{3}{4}I+\tfrac{1}{4}P$. Intermediate options ($P^2$ and the moderately lazy mixtures) receive zero weight. This indicates that, within this family on the Curie-Weiss chain, the best trade-off for the minimax optimization is achieved by combining the slowest $\tfrac{3}{4}I+\tfrac{1}{4}P$ and fastest $P^4$ directions with the base chain $P$.

For the Bernoulli–Laplace level model (Figure~\ref{fig:psg_bl_lazy}), we similarly observe rapid early descent and a stable plateau thereafter as in Figure~\ref{fig:psg_bl_5}. The final objective is approximately $-0.866$ though has not reached convergence given the limited computational budget. The final weight is \[
\mathbf{w}^{(t)} = \Big(
\underbrace{0.33}_{P},\;
\underbrace{0.10}_{P^2},\;
\underbrace{0.00}_{P^4},\;
\underbrace{0.03}_{\frac{1}{4}I+\frac{3}{4}P},\;
\underbrace{0.09}_{\frac{1}{2}(I+P)},\;
\underbrace{0.45}_{\frac{3}{4}I+\frac{1}{4}P}
\Big),
\]
which gives majority of weight on the base transition matrix $P$ and the transition matrix associated with the most ``lazy'' chain $\frac{3}{4}I + \frac{1}{4}P$. This indicates that, within this family on the Bernoulli-Laplace chains, the best trade-off for the minimax optimization is achieved by combining the slowest direction $\tfrac{3}{4}I+\tfrac{1}{4}P$ and $P^2$ direction with the base chain $P$.

We proceed to simulate on higher-dimensional Markov chains associated with both models, with results presented in Figure~\ref{fig:psg_hd}. For these experiments, the family of transition matrices is $\mathcal{B} = \{P, P^2, P^4, P^8, P^{16}\}$ ($n=5$).  For the Bernoulli-Laplace level model, we conduct experiments on $d=10$, while for the Curie-Weiss model, we only choose $d=8$ in order to avoid numerical overflow. We also summarize the objective values of different $\mathbf{w}$'s in Table~\ref{tab:h_hd}.

\begin{figure}[H]
    \centering
    \begin{subfigure}[b]{0.49\textwidth}
        \includegraphics[width=\textwidth]{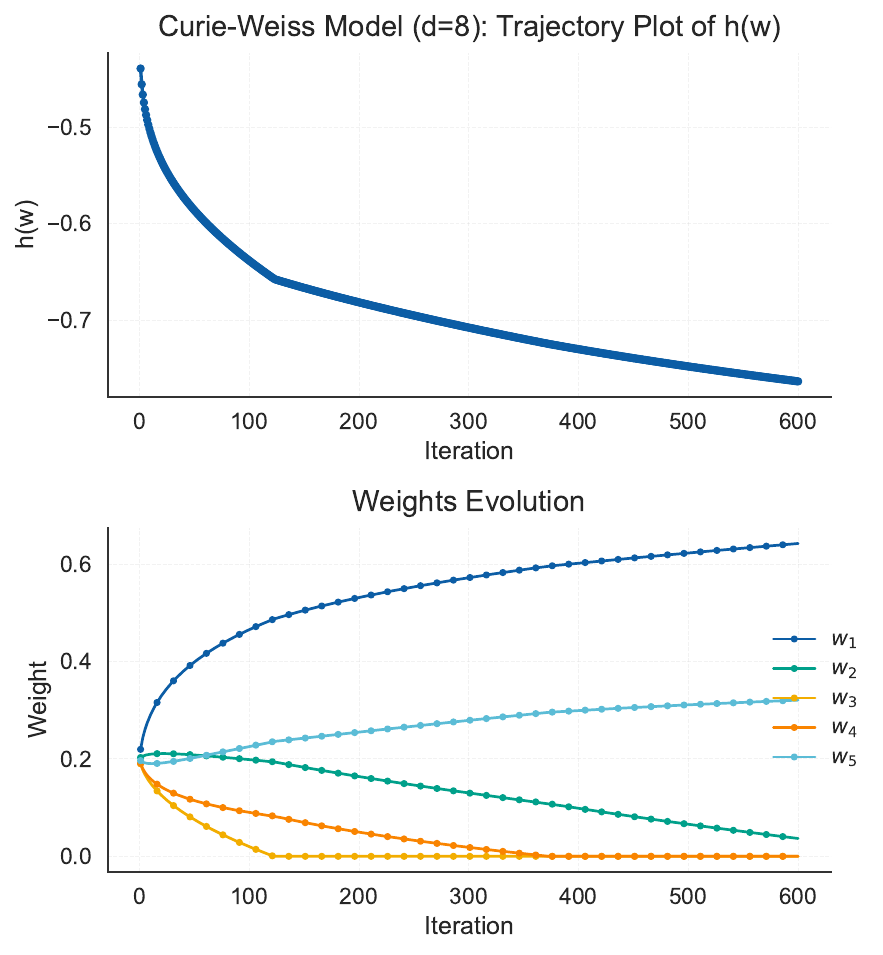}
        \caption{Curie-Weiss model ($d=8$)}
        \label{fig:psg_CW_8}
    \end{subfigure}
    \hfill
    \begin{subfigure}[b]{0.49\textwidth}
        \includegraphics[width=\textwidth]{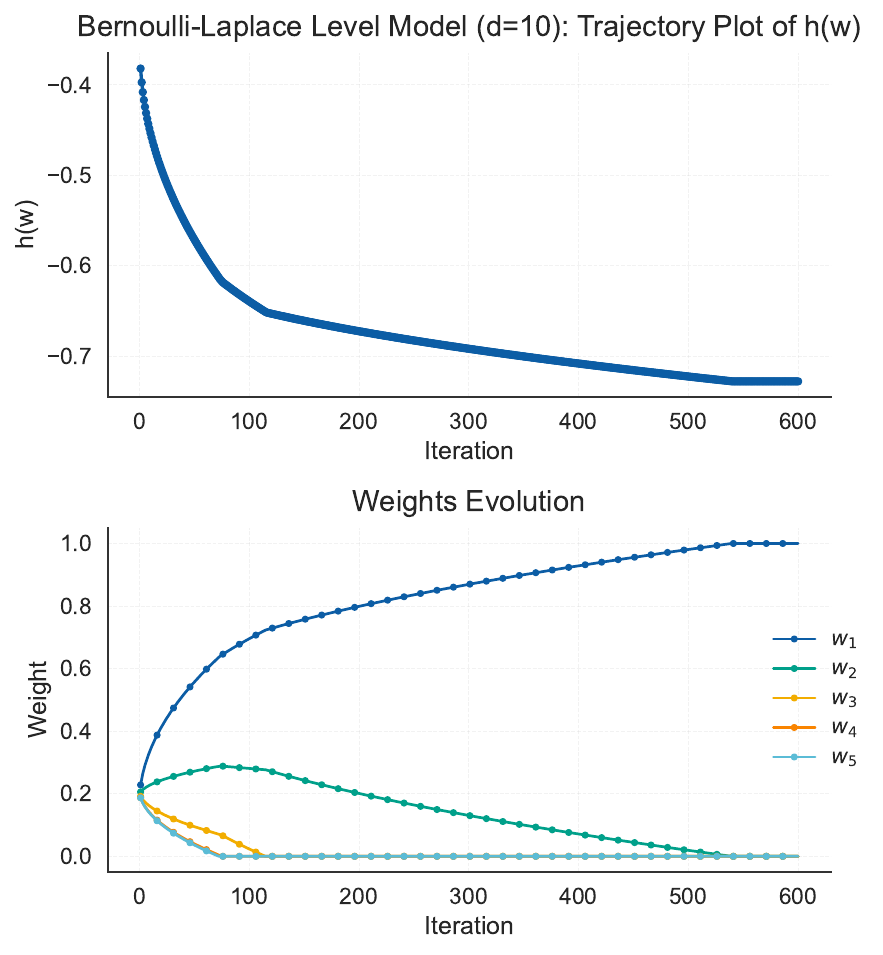}
        \caption{Bernoulli-Laplace level model ($d=10$)}
        \label{fig:psg_BL_10}
    \end{subfigure}
    \caption{Trajectory plots of the projected subgradient algorithm for both models (higher dimension).}
    \label{fig:psg_hd}
\end{figure}

\begin{table}[h]
\centering
\small
\begin{tabular}{l|cc}
\hline
$\mathbf{w}$, $h(\mathbf{w})$ / \textbf{Model} & \textbf{Curie-Weiss} & \textbf{Bernoulli-Laplace} \\
\hline
$\argmin_{i \in \llbracket t \rrbracket} h(\mathbf{w}^{(i)})$ & $(0.64, 0.04, 0.00, 0.00, 0.32)$ & $(1.00, 0.00, 0.00, 0.00, 0.00)$ \\
$\overline{\mathbf{w}}^t$ & $(0.55, 0.13, 0.01, 0.04, 0.27)$ & $(0.83, 0.14, 0.02, 0.01, 0.01)$ \\
$\mathbf{w}^{(0)}$ & $(0.20, 0.20, 0.20, 0.20, 0.20)$ & $(0.20, 0.20, 0.20, 0.20, 0.20)$ \\
$\mathbf{w}_\mathrm{ex}$ & $(1.00, 0.00, 0.00, 0.00, 0.00)$ & $(1.00, 0.00, 0.00, 0.00, 0.00)$ \\
$\mathbf{w}^{(t)}$ & $(0.64, 0.04, 0.00, 0.00, 0.32)$ & $(1.00, 0.00, 0.00, 0.00, 0.00)$ \\
\hline
$\min_{i \in \llbracket t \rrbracket} h(\mathbf{w}^{(i)})$ & $-0.76$ & $-0.73$ \\
$h(\overline{\mathbf{w}}^t)$ & $-0.69$ & $-0.67$ \\
$h(\mathbf{w}^{(0)})$ & $-0.44$ & $-0.38$ \\
$h(\mathbf{w}_\mathrm{ex})$ & $-0.27$ & $-0.73$ \\
$h(\mathbf{w}^{(t)})$ & $-0.76$ & $-0.73$ \\
\hline
\end{tabular}
\caption{Comparison of $h(\mathbf{w})$ values for different weight choices (higher dimension)}
\label{tab:h_hd}
\end{table}

The experiments associated with the Bernoulli-Laplace level model (Figure~\ref{fig:psg_BL_10}) exhibit similar trends as the 5-dimensional example (Figure~\ref{fig:psg_bl_5}), as the objective value $h(\mathbf{w})$ decreases fast at start and then converges slower towards $\mathbf{w}^{(t)} = (1.00, 0.00, 0.00, 0.00, 0.00)$. For the Curie-Weiss model, the 8-dimensional example (Figure~\ref{fig:psg_CW_8}) shows similar convergence trend as the 5-dimensional example (Figure~\ref{fig:psg_cw_5}). However, as the $B$ in Theorem~\ref{thm:subg_ub} is large, we do not obtain the exact converging $\mathbf{w}^*$ with the same computational budget as the Bernoulli-Laplace model.

\subsection{Numerical experiments of Algorithm~\ref{alg:max_max}}

We apply Algorithm~\ref{alg:max_max} to solve problem \eqref{eq:max_min_max} on both the Curie-Weiss and Bernoulli-Laplace models. For both models, we construct a 5-dimensional Markov chain with state space $\mathcal{X} = \{0, 1\}^5$ and $\pi$-stationary transition matrix $P$. We then construct $\mathcal{B} = \{P, P^2, P^4, P^8, P^{16}\}$ so that all matrices in $\mathcal{B}$ share the same stationary distribution $\pi$. We choose the ground set to be $\mathbf{V} = \{V_1, V_2\}$ such that $V_1 = \{1, 2\}$ and $V_2 = \{3, 5\}$. For the inner part, we execute $K=30$ iterations of the projected subgradient algorithm. We summarize the running results of both models in Figure~\ref{fig:two_layer_5}. 

\begin{figure}[H]
    \centering
    \begin{subfigure}{0.95\textwidth}
        \centering
        \includegraphics[width=\linewidth]{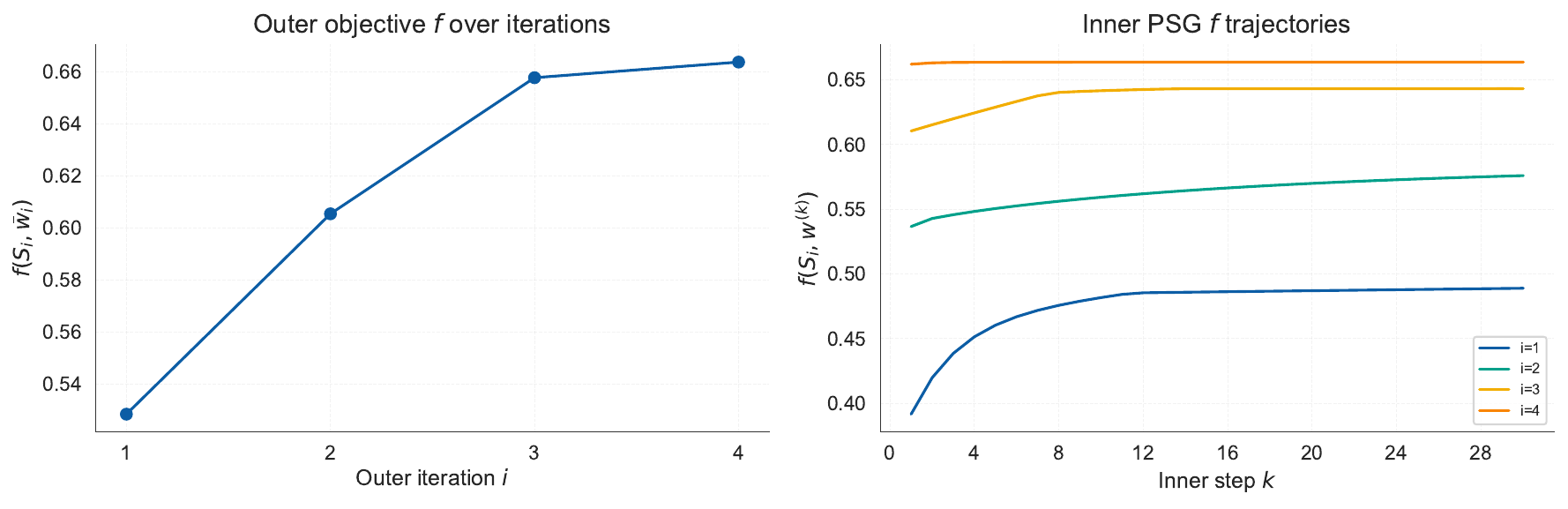}
        \caption{Curie-Weiss model}
        \label{fig:two_layer_CW_5}
    \end{subfigure}
    \begin{subfigure}{0.95\textwidth}
        \centering
        \includegraphics[width=\linewidth]{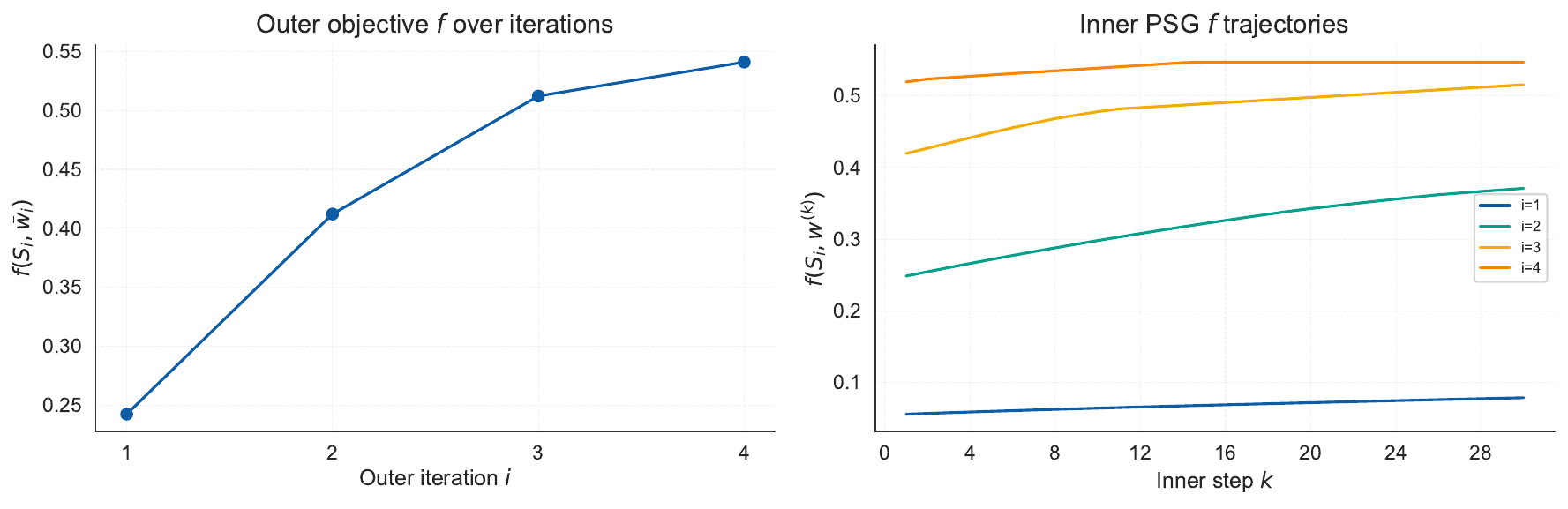}
        \caption{Bernoulli-Laplace level model}
        \label{fig:two_layer_BL_5}
    \end{subfigure}
    \caption{Trajectory plot of Algorithm~\ref{alg:max_max} for both models ($d=5$).}
    \label{fig:two_layer_5}
\end{figure}

For the Curie-Weiss model (Figure~\ref{fig:two_layer_CW_5}), the final weight is $\overline{\mathbf{w}}_l = (0.72, 0.00, 0.00, 0.00, 0.28)$, and the final partition set is $\mathbf{S}_l = \{S_1, S_2\}$, where $S_1 = \{2\}$ and $S_2 = \{3, 5\}$. It shows that after the final round of Algorithm~\ref{alg:max_max}, the resultant weight vector of the max-min-max optimization problem is attained by combining the base transition matrix $P$ and the transition matrix with the highest mixing rate $P^{16}$.

For the Bernoulli-Laplace level model (Figure~\ref{fig:two_layer_BL_5}), the final weight is $\overline{\mathbf{w}}_l = (0.97, 0.03, 0.00, 0.00, 0.00)$, and the final partition set is $\mathbf{S}_l = \{S_1, S_2\}$, where $S_1 = \{2\}$ and $S_2 = \{3, 5\}$. It shows that after the final round of Algorithm \ref{alg:max_max}, the convex hull of family $\mathcal{B}$ concentrates on the base transition matrix $P$.

Similar to the numerical experiments in Section \ref{subsec:num_subg}, we then look into the experiments associated with the family of transition matrices including lazy random walk, precisely, we choose \begin{align*}
    \mathcal{B} = \left\{P, P^2, P^4, \frac{1}{4} I + \frac{3}{4} P, \frac{1}{2} (I+P), \frac{3}{4} I + \frac{1}{4} P \right\}.
\end{align*}
We summarize the results in  Figure~\ref{fig:two_layer_lazy}.

\begin{figure}[H]
    \centering
    \begin{subfigure}{0.95\textwidth}
        \centering
        \includegraphics[width=\linewidth]{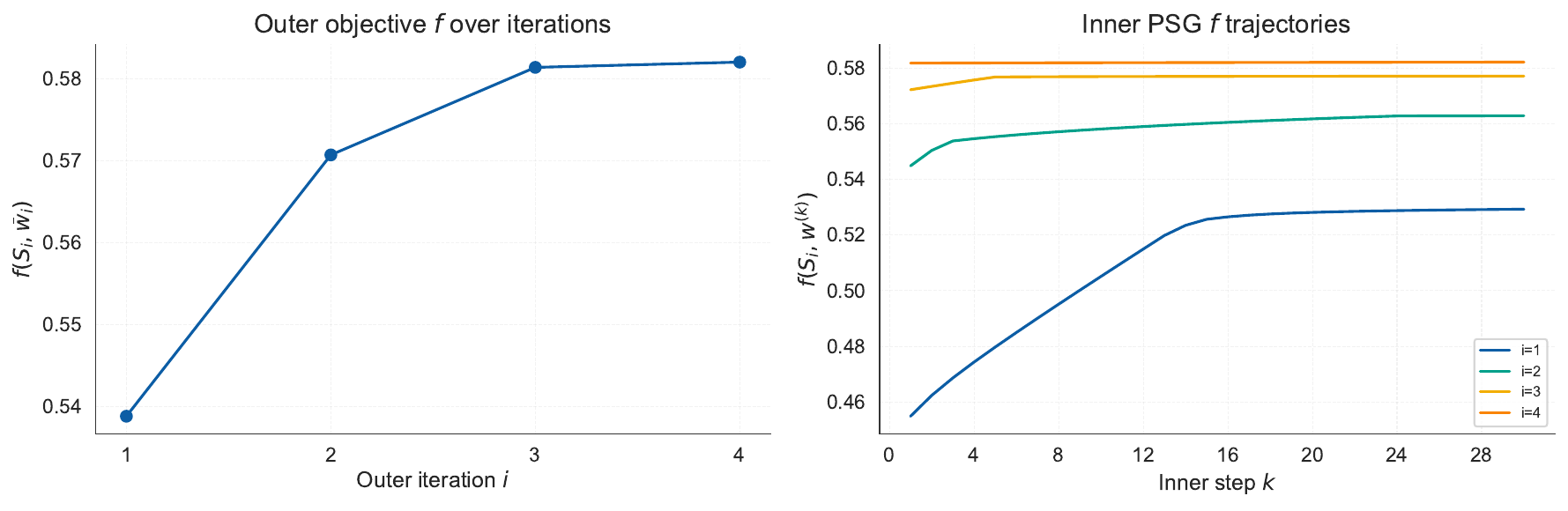}
        \caption{Curie-Weiss model}
        \label{fig:two_layer_CW_lazy}
    \end{subfigure}
    \begin{subfigure}{0.95\textwidth}
        \centering
        \includegraphics[width=\linewidth]{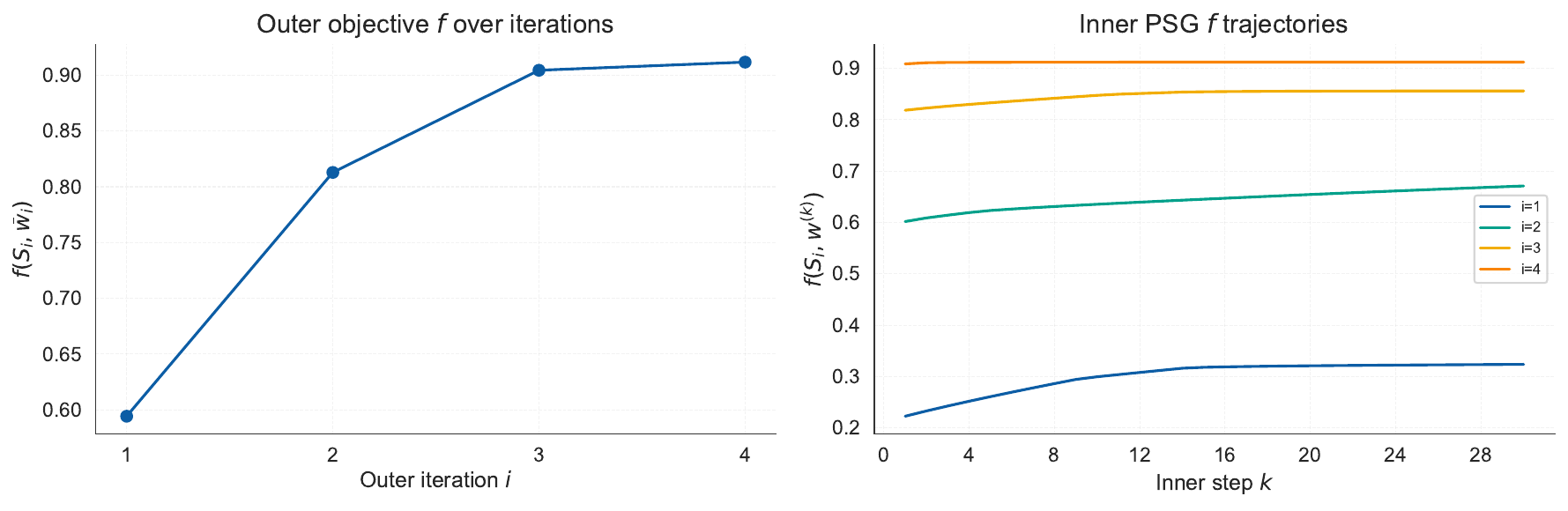}
        \caption{Bernoulli-Laplace level model}
        \label{fig:two_layer_BL_lazy}
    \end{subfigure}
    \caption{Trajectory plot of Algorithm~\ref{alg:max_max} for both models (incl. lazy matrices).}
    \label{fig:two_layer_lazy}
\end{figure}

For the Curie-Weiss model (Figure~\ref{fig:two_layer_CW_lazy}), the final weight is \[
\overline{\mathbf{w}}_l = \Big(
\underbrace{0.37}_{P},\;
\underbrace{0.00}_{P^2},\;
\underbrace{0.33}_{P^4},\;
\underbrace{0.00}_{\frac{1}{4}I+\frac{3}{4}P},\;
\underbrace{0.00}_{\frac{1}{2}(I+P)},\;
\underbrace{0.30}_{\frac{3}{4}I+\frac{1}{4}P}
\Big),
\]
and the final partition set is $\mathbf{S}_l = \{S_1, S_2\}$, where $S_1 = \{2\}$ and $S_2 = \{3, 5\}$. The final weight vector $\overline{\mathbf{w}}_l$ concentrates on three modes, which indicates that the final weight is obtained by combining the slowest $\frac{3}{4}I + \frac{1}{4}P$ and the fastest $P^4$ directions with the base chain $P$.

For the Bernoulli-Laplace level model (Figure~\ref{fig:two_layer_BL_lazy}), the final weight is \[
\overline{\mathbf{w}}_l = \Big(
\underbrace{0.50}_{P},\;
\underbrace{0.00}_{P^2},\;
\underbrace{0.00}_{P^4},\;
\underbrace{0.00}_{\frac{1}{4}I+\frac{3}{4}P},\;
\underbrace{0.00}_{\frac{1}{2}(I+P)},\;
\underbrace{0.50}_{\frac{3}{4}I+\frac{1}{4}P}
\Big),
\]
and the final partition set is $\mathbf{S}_l = \mathbf{V}$, which means that Algorithm~\ref{alg:max_max} selects the whole ground set as the subset. The final output $\overline{\mathbf{w}}_l$ concentrates on two matrices, which indicates that the final result is obtained by averaging the chain with the slowest mixing rate $\frac{3}{4}I + \frac{1}{4}P$ and the base chain $P$.

We proceed to analyze higher-dimensional cases of both models with $d=8$ and cardinality constraint $l=7$, and choose the ground set as $\mathbf{V} = \{V_1, V_2\}$, where $V_1 = \{1, 2, 3, 4\}$ and $V_2 = \{5, 6, 7\}$. We choose the family of the transition probability matrices to be $\mathcal{B} = \{P, P^2, P^4, P^8, P^{16}\}$. For the inner part, we execute $K=150$ iterations of the projected subgradient algorithm. The trajectory plots of both models are summarized in Figure~\ref{fig:two_layer_8}.

\begin{figure}[H]
    \centering
    \begin{subfigure}{0.95\textwidth}
        \centering
        \includegraphics[width=\linewidth]{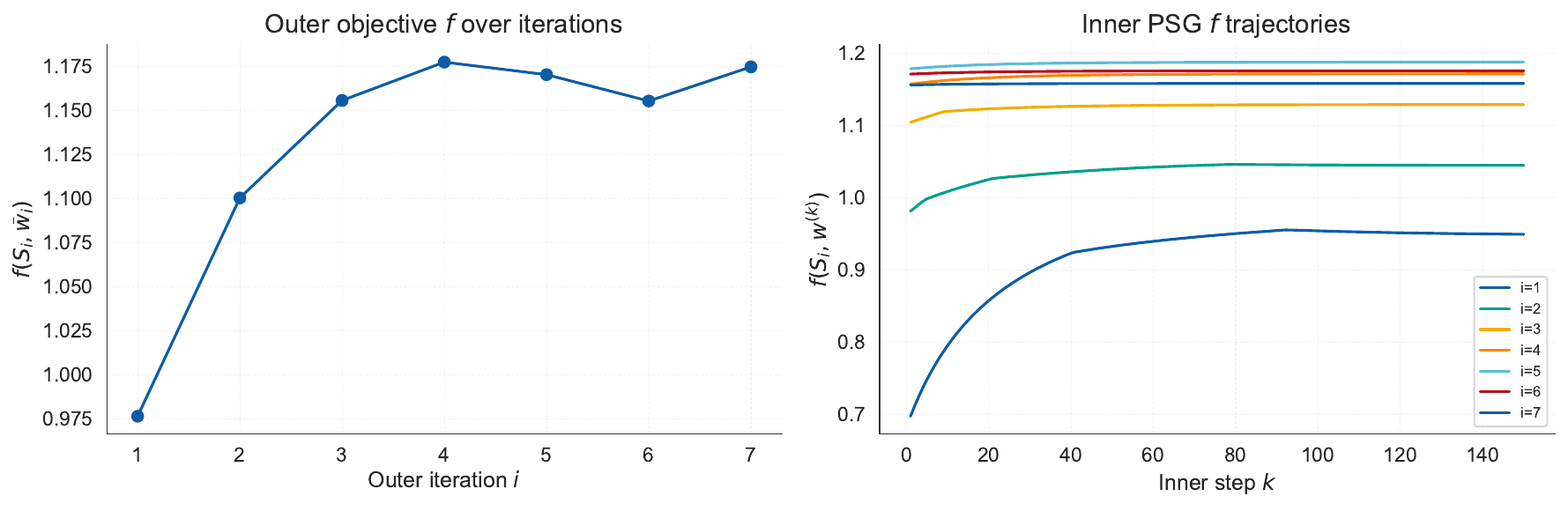}
        \caption{Curie-Weiss model}
        \label{fig:two_layer_CW_8}
    \end{subfigure}
    \begin{subfigure}{0.95\textwidth}
        \centering
        \includegraphics[width=\linewidth]{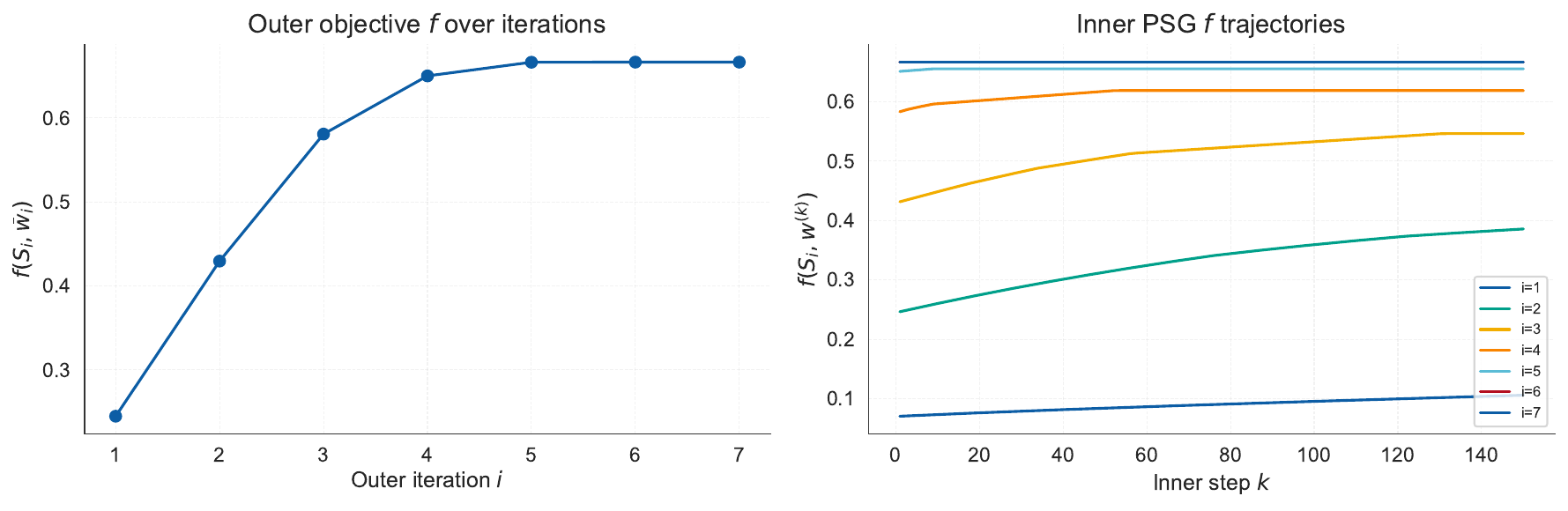}
        \caption{Bernoulli-Laplace level model}
        \label{fig:two_layer_BL_8}
    \end{subfigure}
    \caption{Trajectory plot of Algorithm~\ref{alg:max_max} for both models ($d=8$).}
    \label{fig:two_layer_8}
\end{figure}

For the Curie-Weiss model (Figure \ref{fig:two_layer_CW_8}), the objective value $f(\mathbf{S}_i, \overline{\mathbf{w}}_i)$ is not monotonically non-decreasing, as both the generalized distorted greedy algorithm (Algorithm 3 of \cite{lai2025information}) and the projected subgradient algorithm (Algorithm~\ref{alg:subg}) do not guarantee monotonicity. The final partition set is $\mathbf{S}_l = \mathbf{V}$, which means that the algorithm selects the ground set as the subset. After the final round of Algorithm~\ref{alg:max_max}, the final weight is $\overline{\mathbf{w}}_l = (0.70, 0.00, 0.00, 0.00, 0.30)$, which concentrates on the base transition matrix $P$ and the matrix with fastest mixing $P^{16}$.

For the Bernoulli-Laplace level model (Figure~\ref{fig:two_layer_BL_8}), the final weight is $\overline{\mathbf{w}}_l = (1.00, 0.00, 0.00, 0.00, 0.00)$ and the final partition set is $\mathbf{S}_l = \{S_1, S_2\}$, where $S_1 = \{1, 2, 3\}$ and $S_2 = \{5, 6, 7\}$. It shows that after the final round of Algorithm~\ref{alg:max_max}, the weight of the max-min-max optimization reaches closely to the base transition matrix $P$.

\section*{Declarations}

\paragraph{Funding.}
Michael Choi acknowledges the financial support of the projects A-8001061-00-00, NUSREC-HPC-00001,
NUSREC-CLD-00001, A-0000178-01-00, A-0000178-02-00 and A-8003574-00-00 at National University of Singapore.

\paragraph{Competing interests.}
Both authors have no relevant financial or non-financial interests to disclose.

\paragraph{Data availability.}
No data was used for the research described in the article.

\paragraph{Author contributions.}
Michael Choi and Zheyuan Lai jointly contributed to idea formulation, execution, and manuscript writing. Zheyuan Lai performed the numerical experiments. Michael Choi supervised the project. 

\bibliography{sn-bibliography}

\end{document}